\documentclass[11pt]{amsart}

\usepackage{epigamath}

\usepackage[english]{babel}

\numberwithin{equation}{section}

\usepackage[shortlabels]{enumitem}
\setlist[enumerate,1]{label={\rm(\alph*)}, ref={\rm\alph*}} 

\usepackage{amsxtra}
\usepackage[all]{xypic}
\newtheorem{lem}{Lemma}[section]
\newtheorem{prop}[lem]{Proposition}
\newtheorem{thm}[lem]{Theorem}

\theoremstyle{definition}
\newtheorem{defn}[lem]{Definition}

\theoremstyle{remark}
\newtheorem*{remark}{Remark}

\newcommand{\VR}{\mathcal{O}}
\newcommand{\invo}{\overline{\rule{2.5mm}{0mm}\rule{0mm}{4pt}}}

\renewcommand{\D}{\mathcal{D}}

\newcommand{\V}{\mathcal{V}}
\newcommand{\W}{\mathcal{W}}
\newcommand{\II}{\mathcal{I}}
\newcommand{\LL}{\mathcal{L}}
\newcommand{\M}{\mathcal{M}}
\newcommand{\N}{\mathcal{N}}
\newcommand{\T}{\mathcal{T}}
\newcommand{\Id}{\operatorname{Id}}
\newcommand{\af}{\mathrm{af}}
\newcommand{\pspot}{{\mathfrak{p}}}
\newcommand{\qspot}{{\mathfrak{q}}}
\newcommand{\Mor}{\mathsf{Mor}}
\newcommand{\mult}{\mathsf{mult}}
\newcommand{\can}{\mathsf{can}}

\newcommand{\Mod}{\mathsf{Mod}}
\newcommand{\VB}{\mathsf{Mod}}
\newcommand{\qf}[1]{{\langle{#1}\rangle}} % formes quadratiques
\newcommand{\genxi}{\chi}
\newcommand{\geneta}{\theta}

\newcommand{\resatp}{k(\pspot)}
\newcommand{\resatq}{k(\qspot)}
\newcommand{\resatinf}{k(\infty)}
\newcommand{\Datp}{D(\pspot)}
\newcommand{\Tatp}{T(\pspot)}
\newcommand{\Tatinf}{T(\infty)}
\newcommand{\Lsigatp}{L_\sigma(\pspot)}
\newcommand{\Latp}{L(\pspot)}
\newcommand{\Latinf}{L(\infty)}

\newdir{ (}{{}*!/-5pt/@^{(}} 

\DeclareMathOperator{\Trd}{Trd}

\DeclareMathOperator{\Hom}{Hom}
\DeclareMathOperator{\End}{End}
\DeclareMathOperator{\calHom}{\mathcal{H\!o\!m}}
\DeclareMathOperator{\calEnd}{\mathcal{E\!n\!d}}

\DeclareMathOperator{\ext}{\mathsf{ext}}
\DeclareMathOperator{\res}{\mathsf{res}}

\DeclareMathOperator{\image}{image}
\DeclareMathOperator{\coker}{coker}
\DeclareMathOperator{\Pic}{Pic}
\DeclareMathOperator{\tr}{tr}
\DeclareMathOperator{\Spec}{Spec}
\DeclareMathOperator{\Proj}{Proj}

\DeclareMathOperator{\Skew}{Skew}
\DeclareMathOperator{\PGL}{PGL}

\def\lowsim{\vbox to 0pt{\vss\hbox{$\scriptstyle\sim$}\vskip-2pt}}
\newcommand{\lra}{\longrightarrow}

\EpigaVolumeYear{9}{2025} \EpigaArticleNr{4} \ReceivedOn{April 10, 2023}
\InFinalFormOn{March 20, 2024}
\AcceptedOn{August 7, 2024}

\title{Witt groups of Severi--Brauer varieties\\ and of function fields of conics}
\titlemark{Witt groups of Severi--Brauer varieties}
  
\author{Anne Qu\'eguiner-Mathieu}
\address{Universit\'e Paris 13,  
Sorbonne Paris Nord, 
LAGA - CNRS (UMR 7539),
F-93430 Villetaneuse, France}
\email{queguin@math.univ-paris13.fr}
\author{Jean-Pierre Tignol}
\address{ICTEAM Institute, Box L4.05.01. 
UCLouvain, 
B-1348 Louvain-la-Neuve, Belgium}
\email{jean-pierre.tignol@uclouvain.be}

\authormark{A.~Qu\'eguiner-Mathieu and J.-P.~Tignol}

\AbstractInEnglish{The Witt group of skew-hermitian forms over a
  division algebra $D$ with symplectic involution is shown to be
  canonically isomorphic to the Witt group of symmetric bilinear forms
  over the Severi--Brauer variety of $D$ with values in a suitable
  invertible sheaf. In the special case where $D$ is a quaternion
  algebra, we extend previous work by Pfister and by Parimala on the
  Witt group of conics to set up two five-terms exact sequences
  relating the Witt groups of hermitian or skew-hermitian forms over
  $D$ with the Witt groups of the center, of the function field of the
  Severi--Brauer conic of $D$, and of the residue fields at each
  closed point of the conic.}



\acknowledgement{The second author acknowledges support from the Fonds
  de la Recherche Scientifique--FNRS under grants No.~J.0149.17
  and J.0159.19. Both authors acknowledge support from
  Wallonie--Bruxelles International and the French government in the
  framework of ``Partenariats Hubert Curien'' (Project \emph{Groupes
  alg\'ebriques, anisotropie et invariants cohomologiques}.) }

\begin{document}

%%%%%%%%%%%%%%%%%%%%%%%%%%%%%%%
% Title page
%%%%%%%%%%%%%%%%%%%%%%%%%%%%%%%

\maketitle

\begin{prelims}

\DisplayAbstractInEnglish

\bigskip

\DisplayKeyWords

\medskip

\DisplayMSCclass

\end{prelims}

%%%%%%%%%%%%%%%%%%%%%
% Table of Contents
%%%%%%%%%%%%%%%%%%%%%

\newpage

\setcounter{tocdepth}{1}

\tableofcontents

%%%%%%%%%%%%%%%%%%%%%
% Content begins here
%%%%%%%%%%%%%%%%%%%%%

\section{Introduction}

This paper consists of two parts. In the first part, comprising
Sections~\ref{sec:tauto} and~\ref{sec:SBvar}, we consider a central
division algebra $D$ 
with a symplectic involution $\sigma$ over an arbitrary field $k$ of
characteristic different from~$2$. We make no restriction on the
degree of $D$, which may be an arbitrary even power of~$2$. To the
involution~$\sigma$, we associate an invertible sheaf
  $\LL_\sigma$
on the Severi--Brauer variety $X$ of $D$, 
whose class generates $\Pic(X)$. We relate
skew-hermitian spaces over $(D,\sigma)$ and symmetric
bilinear spaces over $X$ with values in $\LL_\sigma$ by a canonical
isomorphism of Witt groups
\[
  M\colon W^-(D,\sigma)\overset{\lowsim}\lra W(X,\LL_\sigma);
\]
see Theorem~\ref{thm:M}.
The map $M$ is defined
as the composition of the scalar extension map
$\ext_X\colon W^-(D,\sigma)\to W^-(\D,\sigma)$ from $D$ to the Azumaya
algebra $\D=D\otimes_k\VR_X$ over $X$, and a Morita isomorphism
$\Mor\colon W^-(\D,\sigma)\to W(X,\LL_\sigma)$. The injectivity of $M$ is
obtained as a consequence of a theorem of Karpenko \cite{Kar}, and
the surjectivity is derived from Pumpl\"un's description of
$W(X,\LL_\sigma)$ in \cite{Pump2}.
\medbreak

In the second part of the paper, we specialize our discussion to the
case where $D$ is a quaternion algebra. The involution $\sigma$ is
then the canonical involution, and $X$ is a smooth projective conic
without rational points.
The Witt groups $W(X)$ and $W(X,\LL_\sigma)$ satisfy the purity
property (see~\cite[Definition~8.2 and Corollary~10.3]{BW}): 
They embed in the Witt group
$W(F)$ of the function field of $X$, and their images are the kernels
of suitable residue maps. We thus have exact sequences involving the
residue fields $\resatp$ at closed points $\pspot\in
X^{(1)}$:
\begin{equation}
  \label{eq:intro1}
  0\lra W(X)\lra W(F) \overset{\delta}{\lra} \bigoplus_\pspot W(\resatp)
  \quad\text{and}\quad
  0\lra W(X,\LL_\sigma) \lra W(F) \overset{\delta'}\lra \bigoplus_\pspot
  W(\resatp). 
\end{equation}
We compute the cokernels of $\delta$ and $\delta'$ in terms of the
Witt groups $W^+(D,\sigma)$ and $W^-(D,\sigma)$ of hermitian and
skew-hermitian forms over $(D,\sigma)$:
\[
  \coker\delta \simeq W^-(D,\sigma) \quad\text{and}\quad
  \coker\delta'\simeq \ker\left(W(k)\lra W^+(D,\sigma)\right).
\]
These isomorphisms can be interpreted 
in terms of Witt groups of triangulated categories. Indeed, by~\cite[Corollary~92]{Bal} (see also~\cite[Section~8]{BW}), we have 
\[
  \coker\delta \simeq W^1(X) \quad\text{and}\quad
  \coker\delta'\simeq W^1(X,\LL_\sigma).
\]
Hence, we get isomorphisms
\[
 W^1(X)\simeq W^-(D,\sigma) \quad\text{and}\quad
 W^1(X,\LL_\sigma) \simeq \ker\left(W(k)\lra W^+(D,\sigma)\right).
\]
Since $W^1(k)=W^2(k)=0$ and $W^-(D,\sigma)=W^2(D,\sigma)$, the first
isomorphism also follows from Xie's exact sequence, see \cite[Theorem~1.2]{Xie}, 
\[
  \cdots\lra W^1(k) \lra W^1(X) \lra W^2(D,\sigma)\lra W^2(k) \lra\cdots.
\]

As $W(X)$ can be described as the cokernel of an
injective 
transfer map $W^+(D,\sigma)\to W(k)$ (see Proposition~\ref{prop:WC}), 
the description of $\coker\delta$ and $\coker\delta'$, 
together with the isomorphism $M\colon W^-(D,\sigma)\simeq
W(X,\LL_\sigma)$, leads to two strikingly similar exact sequences
\begin{equation}
  \label{eq:intro2}
  0\lra W^+(D,\sigma)\lra W(k)\lra W(F) \overset{\delta}\lra
  \bigoplus_\pspot 
  W(\resatp) \lra W^-(D,\sigma)\lra 0
\end{equation}
and
\begin{equation}
  \label{eq:intro3}
  0\lra W^-(D,\sigma)\lra W(F) \overset{\delta'}\lra \bigoplus_\pspot
  W(\resatp) \lra W(k)\lra W^+(D,\sigma)\lra 0.
\end{equation}
 
In substance, sequences~\eqref{eq:intro2} and \eqref{eq:intro3} are
due to Pfister \cite[Section~7]{Pf}, 
although Pfister does not consider forms over $D$: He
substitutes for $W^+(D,\sigma)$ and $W^-(D,\sigma)$
in~\eqref{eq:intro2} and \eqref{eq:intro3} groups that he defines
specifically for this purpose.

The exactness of \eqref{eq:intro2} and \eqref{eq:intro3} is proved in
Section~\ref{sec:exact}. 
(The exactness of \eqref{eq:intro3} at the middle term has been
established by Parimala \cite[Theorem~5.1]{Pari}, who also has an \emph{ad
  hoc} description of the kernel of $\delta'$ in
\cite[Theorem~5.3]{Pari}.) A delicate part of the argument is to
coherently choose uniformizers and transfer maps $\resatp\to k$ at
each closed point $\pspot$. This issue is addressed in
Section~\ref{sec:restrans}. 
In Section~\ref{sec:octa}, we
set up an exact octagon relating the Witt groups $W^+(D,\sigma)$ and
$W^-(D,\sigma)$ 
to the Witt groups of quadratic or hermitian forms over a maximal
subfield of $D$. This exact octagon, due to Lewis \cite{Lewis}, is a
key technical tool to show that \eqref{eq:intro2} is exact at the
next-to-last term.

For a suitable identification of
$\LL_\sigma$ with an ideal sheaf, it turns
out that the residue maps $\delta$ and $\delta'$ only differ in one
point of degree~$2$, which we designate by $\infty$. As a result,
quadratic forms over $k$ that are split by $\resatinf$ map in $W(F)$ to
forms that lie in the kernel of $\delta'$; hence these forms can be
used to describe skew-hermitian forms over $D$. This idea is a key
ingredient in Becher's proof of the Pfister factor conjecture;  
see \cite{Bec}. It was also used in \cite[Proposition~3.4]{QTPfister} to give
examples of non-similar skew-hermitian forms over a quaternion algebra
that become similar over the function field of its Severi--Brauer
variety. Berhuy uses it in \cite{Ber} to define higher cohomological
invariants of quaternionic skew-hermitian forms. Note that Berhuy's
discussion at the top of p.~442 is flawed: The correspondence between
skew-hermitian forms after scalar extension to the function field and
quadratic forms \emph{does} depend on the choice of
splitting. However, Garrel~\cite[Section~3.1.3]{Garrel:these} has shown how the exact
sequences~\eqref{eq:intro2} and \eqref{eq:intro3} can be used to amend
Berhuy's arguments and expand his result, providing a general method
that produces cohomological invariants of skew-hermitian
forms that depend only on their similarity class.

\subsection*{Notation}
Throughout the paper, we let $D$ denote a central division algebra of
$2$-power degree~$n=2^d\geq2$ over an arbitrary field $k$ of
characteristic 
different from~$2$. We assume $D$ has exponent~$2$ and fix some
symplectic involution $\sigma$ of $D$. Let $X$ be the Severi--Brauer
variety of $n$-dimensional left ideals in $D$. Write $F=k(X)$
for its 
function field and $\VR_X$ for its structure sheaf. For each point
$\pspot$ on $X$, we write $\VR_\pspot$ for the local ring at $\pspot$
and $\resatp$ for its residue field. We let $\D=
D\otimes_k\VR_X$ denote the Azumaya algebra over $X$ obtained by
scalar extension to $\VR_X$. Its stalk and fiber at a point $\pspot$
are
\[
\D_\pspot = D\otimes_k\VR_\pspot \quad\text{and}\quad
\Datp= D\otimes_k\resatp.
\]
From Section~\ref{sec:octa} onward, $D$ is assumed to be a quaternion
algebra. The involution $\sigma$ is therefore the canonical
conjugation involution $\invo$; we often omit it from the notation and
write simply $W(D)$, $W^-(D)$ for $W^+(D,\sigma)$, $W^-(D,\sigma)$.

\subsection*{Acknowledgments}
The authors are grateful to A.~Merkurjev for the proof of
Proposition~\ref{prop:newT} and to the referee for suggestions that
allowed them to streamline the arguments in Section~\ref{sec:SBvar}.

\section{Locally free sheaves on Severi--Brauer varieties}
\label{sec:tauto}

In this section, we define on the Severi--Brauer variety $X$ a locally
free sheaf $\T$ of rank~$n$, which is the main tool for the Morita
equivalence developed in the next section. We use it to associate to
the symplectic involution $\sigma$ the generator $\LL_\sigma$ of
$\Pic(X)$ in which the symmetric bilinear spaces over $X$ we consider
take their values.

Let $T$ be the generic point of $X$, which is an $n$-dimensional left
ideal in the split algebra $D_F$ obtained from~$D$ by scalar extension
to the function field $F$ of $X$. The sheaf of $\VR_X$-modules $\T$ is
defined as the intersection of $\D$ with $T$ in $D_F$ (viewing $T$ and
$D_F$ as constant sheaves):
\begin{equation}
  \label{eq:defT}
  \T=T\cap\D\subset D_F.
\end{equation}
Since $T$ is a left ideal in $D_F$, it is clear that $\T$ is a sheaf
of left $\D$-modules.
The main properties of the sheaf~$\T$ are given in the
next proposition, using the following notation: For $\ell$ an
arbitrary field extension of $k$, let 
$X_\ell=X\times\Spec(\ell)$ be the $\ell$-variety obtained from $X$ by
base change, and let $p\colon X_\ell\to X$ be the projection map. For
any $\VR_X$-module $\M$, we let
$\M_\ell=p^*(\M)$ be the inverse image of $\M$; if $\ell$ is a
finite extension of $k$ and $\N$ is an $\VR_{X_\ell}$-module, we let
$\tr_{\ell/k}(\N)=p_*(\N)$ be the direct 
image of $\N$.

\begin{prop}
  \label{prop:newT}\leavevmode
  \begin{enumerate}
  \item\label{p:nT-a}
    The sheaf $\T$ is a locally free $\VR_X$-module of rank~$n$.
  \item\label{p:nT-b}
    If\, $\ell$ is a splitting field of\, $D$, every $\ell$-algebra
    isomorphism $D_\ell\simeq\End_\ell V$ with $V$ an $n$-dimensional
    $\ell$-vector space induces an isomorphism of sheaves
    \begin{equation}
      \label{eq:Ts}
      \T_\ell \,\simeq  V\otimes_\ell \VR_{X_\ell}(-1)\simeq
      \VR_{X_\ell}(-1)^{\oplus n}.
    \end{equation}
  \item\label{p:nT-c}
    The canonical homomorphism $D\to\End\T$ arising from the left
    $\D$-module structure on $\T$ yields an identification $D=\End\T$;
    hence $\T$ is an indecomposable locally free sheaf.
  \item\label{p:nT-d}
    For every maximal subfield $\ell$ of $D$, there is an isomorphism
    of sheaves \[\T\simeq\tr_{\ell/k}\left(\VR_{X_\ell}(-1)\right).\]
  \end{enumerate}
\end{prop}

\begin{proof}
  We first prove\footnote{We are indebted to A.~Merkurjev for
    suggesting this proof 
    to us.}
  \eqref{p:nT-b}, as \eqref{p:nT-a} follows by base change. Let $\ell$ be a
  splitting field of $D$, and fix an $\ell$-algebra isomorphism to
  identify $D_\ell=\End_\ell(V)=V\otimes_\ell V^*$ for some
  $n$-dimensional $\ell$-vector space $V$. Then $X_\ell$ is identified
  with the projective space $\mathbb{P}(V^*)$, viewing each line
  $d\subset V^*$ as the $n$-dimensional left ideal $V\otimes_\ell  d$.
    Pick an $\ell$-base $v_1$, \ldots, $v_n$ of $V$, so
    $X_\ell=\Proj(\ell[v_1,\ldots, v_n])$, and let $U\subset X_\ell$
    be the open subscheme defined by $v_n\neq0$, so
    $U=\Spec(\ell[v_1v_n^{-1},\ldots, v_{n-1}v_n^{-1}])$. The field
    $\ell(X_\ell)=\ell(U)$ is the rational function field
    $\ell(v_1v_n^{-1},\ldots,v_{n-1}v_n^{-1})$, and the module of
    sections of $\D_\ell$ over $U$ is
    \[
      \D_\ell(U)=V\otimes_\ell V^*\otimes_\ell\VR_{X_\ell}(U) =
      V\otimes_\ell V^*\otimes_\ell\ell[v_1v_n^{-1},\ldots,
      v_{n-1}v_n^{-1}] \subset 
      V\otimes_\ell V^*\otimes_\ell\ell(X_\ell)=D_{\ell(X_\ell)}.
    \]
    On the other hand, the $\ell(X_\ell)$-rational point
      induced by base change from the generic point of $X$ is the
    line 
    $S=\genxi\cdot \ell(X_\ell)\subset V^*\otimes_\ell\ell(X_\ell)$,
    where
    \[
     \genxi=\sum_{i=1}^{n} v_i^*\otimes v_iv_n^{-1}
     \in 
     V^*\otimes_\ell\ell(X_\ell).
   \]
    Viewed as a left ideal in
    $D_{\ell(X_\ell)}$, this point is $T_\ell=V\otimes_\ell S$. Since
    $S\cap(V^*\otimes_\ell\VR_{X_\ell}(U))$ is the
    $\VR_{X_\ell}(U)$-span of $\genxi$, it follows that
    \begin{equation}
      \label{eq:newT}
      \T_\ell(U)=T_\ell\cap \D_\ell(U) = V\otimes_\ell
      \genxi\cdot\VR_{X_\ell}(U).
    \end{equation}
    Now, there is a canonical embedding
    $\VR_{X_\ell}(-1)\to 
    V^*\otimes_\ell \VR_{X_\ell}$ which on $U$ maps $v_n^{-1}$ to
    $\genxi$. 
    Tensoring with $V$ yields an
    embedding $V\otimes_\ell\VR_{X_\ell}(-1)\to V\otimes_\ell
    V^*\otimes_\ell \VR_{X_\ell}=\D_\ell$. The module of sections over
    $U$ of the image of this embedding is exactly $\T_\ell(U)$. The
    same holds for every 
    open subscheme in the standard affine cover of~$X_\ell$; hence we
    may identify $\T_\ell=V\otimes_\ell\VR_{X_\ell}(-1)$,
    proving~\eqref{eq:Ts}.

    \eqref{p:nT-c}~ Continuing with the same notation, consider
    $\T_\ell\spcheck=\calHom(\T_\ell,\VR_{X_\ell})$, the dual sheaf 
    of $\T_\ell$. From~\eqref{eq:Ts} it follows
    that $\T_\ell\spcheck\simeq V^*\otimes_{\ell} \VR_{X_\ell}(1)$,
   hence 
 \[
  \calEnd\T_\ell = \T_\ell\otimes\T_\ell\spcheck \simeq  
\End_{\ell}(V) \otimes_{\ell}\VR_{X_\ell}(0).
  \]
  This shows that $\dim_{\ell}(\End\T_\ell)=n^2$, hence
  $\dim_k(\End\T)=n^2$. The canonical map $D\to\End\T$ is injective
  since $D$ is a division algebra; hence it is an isomorphism by 
  dimension count.

  \eqref{p:nT-d}~ Now, let $\ell$ be a maximal subfield of $D$. Since $D=\End\T$,
  the locally free $\VR_X$-module $\T$ has an $\ell$-structure, and
  $\T\simeq\tr_{\ell/k}(\N)$ for some irreducible locally free
  $\VR_{X_\ell}$-module $\N$ by \cite[Theorem~1.8]{AEJ}. Lemma~1.4 of
  \cite{AEJ} shows 
  that the $\VR_{X_\ell}$-module $\N$ is a direct summand of
  $\T_\ell$. By~\eqref{eq:Ts} it follows that
  $\N\simeq\VR_{X_\ell}(-1)$. 
\end{proof}

\begin{remark}
  For $\ell$ a Galois extension of $k$ that splits $D$, it follows
  from Proposition~\ref{prop:newT} that $\T$ can be obtained by Galois
  descent from $V\otimes_\ell\VR_{X_\ell}(-1)$ by using the 
  cocycle with values in $\PGL(V)$ that twists $\End_\ell(V)$ into
  $D$. Therefore, $\T$ can be identified with the sheaf $J$ defined
  by Quillen~\cite[Section~8.4]{Q} in his computation of the $K$-theory of
  Severi--Brauer varieties.
\end{remark}

In order to define the invertible $\VR_X$-module $\LL_\sigma$ attached
to the involution $\sigma$, we start with some observations on 
split central simple algebras, which will be applied to the scalar
extension of $D$ to the residue fields at points of $X$.

Let $A$ be a split central simple algebra of even degree $n=2m$ over
an arbitrary field $E$, and let $\sigma_A$ be a symplectic involution
on $A$. We let
\[
  \Skew(\sigma_A)=\{x\in A\mid \sigma_A(x)=-x\}
\]
and write $\Trd\colon A\to E$ for the reduced trace map. Recall that
the bilinear form $\Trd(xy)$ is nonsingular; hence for every nonzero
$x\in A$, there exists a $y\in A$ such that $\Trd(xy)=1$.

\begin{lem}
  \label{lem:tauto}
  Let $I\subset A$ be an $n$-dimensional left
  ideal, let $\sigma_A(I)=\{\sigma_A(\xi)\mid\xi\in I\}$ be its
  conjugate $n$-dimensional right ideal, and denote by $J$ the intersection $J=I\cap
  \sigma_A(I)$. Then 
  $\dim_EJ=1$ and $J=I\cap\Skew(\sigma_A)$; moreover, for $\lambda\in
  J$ and $\mu\in A$, we have 
  \begin{equation}
    \label{eq:tauto1}
    \sigma_A(\lambda)=-\lambda,\quad \lambda^2=0,\quad\text{and}\quad
    \lambda\mu\lambda=\Trd(\lambda\mu)\lambda.
  \end{equation}
  Multiplication in $A$ defines an isomorphism of $E$-vector
  spaces
  \[
  \mult\colon \sigma_A(I)\otimes_{A} I\lra J, \quad
  \sigma_A(\xi)\otimes\eta \longmapsto \sigma_A(\xi)\eta.
  \]
  Moreover, there is a canonical isomorphism of $A$-bimodules
  \[
  \can\colon I\otimes_E\sigma_A(I) \overset{\lowsim}{\lra} J\otimes_EA
  \]
  defined as follows: Pick $\lambda\in J$ and $\mu\in A$ such that
  $\Trd(\lambda\mu)=1$, and let
  \[
  \can(\xi\otimes\sigma_A(\eta))=\lambda\otimes\xi\mu\,\sigma_A(\eta)
  \quad\text{for $\xi$, $\eta\in I$.}
  \]
\end{lem}

\begin{proof}
  Fix some representation $A=\End_EV$ for some $n$-dimensional
  $E$-vector space $V$. The involution $\sigma_A$ is then adjoint to
  some 
  alternating bilinear form $b$ on $V$, and the ideal $I$ is the set
  of linear operators 
  that vanish on some hyperplane $H$; its conjugate $\sigma_A(I)$ is
  the set of operators that map $V$ into $H^\perp$, the orthogonal of
  $H$ for the form $b$. Therefore, an operator lies in $J$ if and only
  if its kernel contains $H$ and its image is in the $1$-dimensional
  subspace $H^\perp$. It follows that $\dim J=1$. Moreover, the image
  of each $\lambda\in J$ lies in $H^\perp$, hence in $H$ since $b$ is
  alternating, and therefore $\lambda^2=0$.

  To complete the proof of~\eqref{eq:tauto1}, pick $v\in V\setminus
  H$. Every vector $x\in V$ has the form $x=v\alpha+u$ for some
  $\alpha\in E$ and $u\in H$. For $\lambda\in J$, we have
  $\lambda(u)=0$ and $\lambda(v)\in H^\perp$; hence for
  $x'=v\alpha'+u'$ with $\alpha'\in E$ and $u'\in H$, 
  \[
  b(\lambda(x),x')=b(\lambda(v)\alpha,v\alpha') \quad\text{and}\quad
  b(x,\lambda(x'))=b(v\alpha,\lambda(v)\alpha').
  \]
  Since $b$ is alternating, $b(v,\lambda(v))=-b(\lambda(v),v)$, and it
  follows that $b(\lambda(x),x') = 
  -b(x,\lambda(x'))$ for all $x$, $x'\in V$, hence
  $\sigma_A(\lambda)=-\lambda$. 

  Now, take $\mu\in A$. As $\mu\lambda$ vanishes on $H$, it follows
  that $\mu\lambda(v)=v\Trd(\mu\lambda) +u$ for some $u\in H$, hence
  $\lambda\mu\lambda(v)=\lambda(v)\Trd(\mu\lambda)$. Moreover,
  $\lambda\mu\lambda(u)=\lambda(u)=0$ for all $u\in H$; hence
  $\lambda\mu\lambda = \Trd(\lambda\mu)\lambda$ since $V$ is spanned
  by $v$ and $H$.

  It follows from~\eqref{eq:tauto1} that $J\subset\Skew(\sigma_A)$,
  hence $J\subset I\cap\Skew(\sigma_A)$. For the reverse inclusion, it
  suffices to observe that if $\lambda\in I\cap\Skew(\sigma_A)$, then
  $\lambda=-\sigma_A(\lambda)\in \sigma_A(I)$, hence $\lambda\in
  I\cap\sigma_A(I)$. Therefore, $J=I\cap\Skew(\sigma_A)$.

  Because $I$ is a left ideal and $\sigma_A(I)$ is a right ideal, we
  have $\sigma_A(I)\cdot I\subset I\cap\sigma_A(I)$; hence
  multiplication defines an $E$-linear map $\mult\colon
  \sigma_A(I)\otimes_A I\to J$. To show that this map is onto, note
  that for any nonzero $\lambda\in J$, there exists a $\mu\in A$ such
  that $\Trd(\lambda\mu)=1$. By~\eqref{eq:tauto1}, it follows that
  \begin{equation}
    \label{eq:tauto2}
    \lambda=\Trd(\lambda\mu)\lambda =
    \lambda\mu\lambda.
  \end{equation}
  Since $\lambda\in I\cap\sigma_A(I)$, this equation shows that
  $\lambda\in\sigma_A(I)\cdot I$; hence $\mult$ is surjective. To see
  that 
  it is injective, pick $\lambda\in J$ and $\mu\in A$ as
  above. Since $\lambda$ is nonzero and lies in $I$, we have
  $I=A\lambda$; hence every element in $\sigma_A(I)\otimes_A I$ can be
  written in the form $\xi\otimes\lambda$ for some
  $\xi\in\sigma_A(I)$. If $\xi\lambda=0$, then using~\eqref{eq:tauto2},  
  we get 
  \[
  \xi\otimes\lambda=\xi\otimes\lambda\mu\lambda= \xi\lambda\mu\otimes
  \lambda=0.
  \]
  Therefore, $\mult$ is injective.

  We next consider the map $\can$, which is clearly a homomorphism of
  $A$-bimodules. We first show that it is canonical, \textit{i.e.}, that it
  does not depend on the choice of $\lambda$ and $\mu$. Suppose
  $\lambda$, $\lambda'\in J$ and $\mu$, $\mu'\in A$ are such that
  $\Trd(\lambda\mu)=\Trd(\lambda'\mu')=1$. Because $\dim J=1$, there
  exists an $\alpha\in E^\times$ such that $\lambda'=\alpha\lambda$,
  hence $\Trd(\lambda\mu')=\alpha^{-1}$, so \eqref{eq:tauto1} and
  \eqref{eq:tauto2} yield 
  \begin{equation}
    \label{eq:tauto3}
    \lambda\mu'\lambda=\alpha^{-1}\lambda=\alpha^{-1}\lambda\mu\lambda.
  \end{equation}
  Since $I=A\lambda$, for all $\xi$, $\eta\in I$,  we may find
  $\xi_1$, $\eta_1\in A$ such that
  \[
  \xi=\xi_1\lambda\quad\text{and}\quad \eta=\eta_1\lambda.
  \]
  Then, by~\eqref{eq:tauto2},
  \[
  \xi\mu'\,
  \sigma_A(\eta)=-{\xi_1}\lambda\mu'\lambda\,\sigma_A({\eta_1}) = 
  -\alpha^{-1}{\xi_1}\lambda\mu\lambda\,\sigma_A({\eta_1})=\alpha^{-1}
  \xi\mu\sigma_A(\eta). 
  \]
  Therefore, 
  \[
  \lambda'\otimes\xi\mu'\,\sigma_A(\eta)
  = \lambda\otimes \xi\mu\,\sigma_A(\eta).
  \]
  It follows that $\can$ is canonical, and it remains to prove that
  it is 
  bijective. Since $A$ is a simple algebra, we have $A\lambda
  A=A$; hence to prove surjectivity,  it suffices to show that
  $\lambda\otimes\xi\lambda\eta$ lies in the image of $\can$ for all
  $\xi$,~$\eta\in A$. By~\eqref{eq:tauto2}, we have
  \[
  \lambda\otimes\xi\lambda\eta=\lambda\otimes\xi\lambda\mu\lambda\eta
  = \can(\xi\lambda\otimes\lambda\eta); 
  \]
  hence $\can$ is surjective. It is therefore also injective by
  dimension count. 
\end{proof}

Our first application of Lemma~\ref{lem:tauto} is to $A=D_F$, the
split algebra obtained from $D$ by scalar extension to the function
field of $X$. Taking for $I$ the generic point $T$ of $X$, we let
\[
L_\sigma = T\cap\sigma(T)=T\cap\Skew(\sigma).
\]
Lemma~\ref{lem:tauto} shows that $L_\sigma$ is an $F$-vector space of
dimension~$1$ and yields canonical isomorphisms
\begin{equation}
\label{eq:tautodef}
\mult\colon\sigma(T)\otimes_{D_F} T \overset{\lowsim}{\lra} L_\sigma
\quad\text{and}\quad 
\can\colon T\otimes_F\sigma(T)\overset{\lowsim}{\lra} L_\sigma\otimes_FD_F.
\end{equation}
The sheaf of $\VR_X$-modules $\LL_\sigma$ is defined as the
intersection of $\D$ with $L_\sigma$ in $D_F$, viewing $L_\sigma$ and
$D_F$ as constant sheaves:
\[
  \LL_\sigma=L_\sigma\cap\D\subset D_F.
\]

\begin{prop}
  \label{prop:newL}
  The sheaf $\LL_\sigma$ is an invertible $\VR_X$-module such that
  $(\LL_\sigma)_\ell\simeq \VR_{X_\ell}(-2)$ for every splitting field
  $\ell$ of\, $D$; hence $\LL_\sigma$ generates the Picard group
  $\Pic(X)$. Moreover, there exist isomorphisms of $\VR_X$-modules
  \[
    \mult_X\colon\sigma(\T)\otimes_{\D}\T \overset{\lowsim}{\lra}
    \LL_\sigma \quad\text{and}\quad
    \can_X\colon\T\otimes_{\VR_X}\sigma(\T) \overset{\lowsim}{\lra}
    \LL_\sigma\otimes_{\VR_X}\D
  \]
  that restrict on the generic fiber to the isomorphisms
  of~\eqref{eq:tautodef}. 
\end{prop}

\begin{proof}
  Let $\ell$ be a splitting field of $D$. As in the proof of
  Proposition~\ref{prop:newT}, we identify $D_\ell=\End_\ell
  V=V\otimes_\ell V^*$ for some $n$-dimensional $\ell$-vector space
  $V$, hence also $X_\ell=\mathbb{P}(V^*)$. Writing again $\sigma$ for
  the scalar extension of $\sigma$ to $D_\ell$, we know
  from~\cite[Equation~(4.2)]{BoI} that $\sigma$ is the adjoint involution of a
  nonsingular alternating bilinear form $b$ on $V$. Let $m=\frac n2$
  and fix a symplectic base $(u_i,\,w_i)_{i=1}^m$ of $V$; thus, for
  $i$, $j=1$, \ldots, $m$, 
  \[
    b(u_i,w_i)=1=-b(w_i,u_i),\quad
    b(u_i,u_j)=b(w_i,w_j)=0,
  \]
  and
  \[
    b(u_i,w_j)=0=b(w_j,u_i)\quad\text{if $i\neq j$.}
  \]
  It follows that for $i$, $j=1$, \ldots, $m$, 
  \begin{equation}
    \label{eq:newL}
    \begin{split}
      \sigma(u_i\otimes u_j^*)=w_j\otimes w_i^*,\quad&
    \sigma(u_i\otimes w_j^*)=-u_j\otimes w_i^*,\\
    \sigma(w_i\otimes u_j^*)=-w_j\otimes u_i^*,\quad&
    \sigma(w_i\otimes w_j^*)=u_j\otimes u_i^*.
    \end{split}
  \end{equation}
  Let $U\subset X_\ell$ be the open subscheme defined by $w_m\neq0$; 
  hence
  \[
    \VR_{X_\ell}(U)=\ell\left[u_1w_m^{-1},\ldots, u_mw_m^{-1},\,
    w_1w_m^{-1},\ldots, w_{m-1}w_m^{-1}\right].
  \]
  As in the proof of Proposition~\ref{prop:newT}, consider
  \[
    \genxi=\sum_{i=1}^m \left(u_i^*\otimes u_iw_m^{-1} + w_i^*\otimes
    w_iw_m^{-1}\right) \in V^*\otimes_\ell\VR_{X_\ell}(U),
  \]
  which has the property that $\chi\cdot\ell(X_\ell)$ is the
    $\ell(X_\ell)$-rational point induced by base change from the
    generic point of $X$. We saw in the
  proof of Proposition~\ref{prop:newT} (see~\eqref{eq:newT}) that
  \[
    \T_\ell(U)=V\otimes_\ell \genxi\cdot\VR_{X_\ell}(U) \subset
    V\otimes_\ell V^*\otimes_\ell \VR_{X_\ell}(U); 
  \]
  hence every
  element $t\in\T_\ell(U)$ can be written in the form
  \begin{align*}
    t &= \sum_{i=1}^m\left(u_i\otimes\genxi f_i+w_i\otimes\genxi g_i\right)\\
    &= \sum_{i,j=1}^m\left(u_i\otimes u_j^*\otimes u_jw_m^{-1}f_i +
      u_i\otimes w_j^*\otimes w_jw_m^{-1}f_i\right.
    \\
    &\left.\hphantom{\sum_{i,j=1}^m(u_i}+ w_i\otimes u_j^*\otimes
      u_jw_m^{-1}g_i + w_i\otimes w_j^*\otimes w_jw_m^{-1}g_i\right)
  \end{align*}
  for some $f_1$, \ldots, $f_m$, $g_1$, \ldots, $g_m\in
  \VR_{X_\ell}(U)$. A straightforward computation
  using~\eqref{eq:newL} shows that $\sigma(t)=-t$ holds if and only if
  for all $i$, $j=1$, \ldots, $m$, 
  \[
    f_i=w_iw_m^{-1}f_m \quad\text{and}\quad
    g_i=-u_iw_m^{-1}f_m.
  \]
  Therefore, $(\LL_\sigma)_\ell(U)=\T_\ell(U)\cap\Skew(\sigma)$ is the
  $\VR_{X_\ell}(U)$-span of 
  the following element:
   \begin{multline*}
    \zeta=\sum_{i,j=1}^m \left(u_i\otimes u_j^*\otimes u_jw_iw_m^{-2} +
    u_i\otimes w_j^*\otimes w_iw_jw_m^{-2}\right.
    \\
  \left.  - w_i\otimes u_j^*\otimes
    u_iu_jw_m^{-2} - w_i\otimes w_j^*\otimes u_iw_jw_m^{-2}\right)\in V\otimes_\ell V^*\otimes_\ell \VR_{X_\ell}(U).
  \end{multline*}
  Under the identification $(V\otimes_\ell\VR_{X_\ell}(U))
  \otimes_{\VR_{X_\ell}(U)}(V^*\otimes_\ell\VR_{X_\ell}(U))
  = V\otimes_\ell V^*\otimes_\ell \VR_{X_\ell}(U)$, this element
  $\zeta$ is 
  the tensor product $\zeta=\geneta\otimes\genxi$,  where
  $\geneta=\sum_{i=1}^m(u_i\otimes w_iw_m^{-1} - w_i\otimes u_iw_m^{-1})
  \in V\otimes_\ell\VR_{X_\ell}(U)$ is the element such that
  $b(\geneta,\rho)=\genxi(\rho)$ for all $\rho\in
  V\otimes_\ell\VR_{X_\ell}(U)$.

  Now, there is a canonical embedding $\VR_{X_\ell}(-2) \to
  V\otimes_\ell V^*\otimes_\ell\VR_{X_\ell}$ that on $U$ maps
  $w_m^{-2}$ to $\zeta$, and the
  computation above shows that the module of sections over $U$ of the
  image of this embedding is exactly $(\LL_\sigma)_\ell(U)$. The same
  holds for every open subscheme in the standard affine cover of
  $X_\ell$; hence $(\LL_\sigma)_\ell\simeq\VR_{X_\ell}(-2)$. By base
  change, 
  it follows that $\LL_\sigma$ is an invertible $\VR_X$-module.

  We next define the morphism $\can_X$. Let $U\subset X$ be an affine
  open 
  subscheme on which $\LL_\sigma(U)$ is a free $\VR_X(U)$-module, and
  let $\lambda\in\LL_\sigma(U)$ be a base of $\LL_\sigma(U)$. For each
  point $\pspot$ of $U$, the germ $\lambda_\pspot$ is an
  $\VR_\pspot$-base of the stalk $(\LL_\sigma)_\pspot$, and its image
  $\overline\lambda_\pspot$ in $\Datp$ is a $\resatp$-base of the
  fiber of $\LL_\sigma$ at $\pspot$. Since the bilinear form $\Trd(xy)$ on
  $\Datp$ is nonsingular, the linear form
 $\Trd(\overline\lambda_\pspot\,\text{\textvisiblespace})\colon
 \Datp\to \resatp$  is 
  surjective. From Nakayama's lemma, it follows that the linear form
 $\Trd(\lambda_\pspot\,\text{\textvisiblespace})\colon
 \D_\pspot\to\VR_\pspot$ is 
  surjective. This holds for every point $\pspot$ of $U$; hence the linear form $\Trd(\lambda\,\text{\textvisiblespace})\colon \D(U)\to
  \VR_X(U)$ is surjective. It follows that there exists a $\mu\in\D(U)$ such that
  $\Trd(\lambda\mu)=1$. We may then define the $\VR_X(U)$-module
  homomorphism $\can_U\colon \T(U)\otimes_{\VR_X(U)}\sigma(\T)(U) \to
  \LL_\sigma(U)\otimes_{\VR_X(U)}\D(U)$ by mapping
  $t\otimes\sigma(t')$ to $\lambda\otimes t\mu\sigma(t')$ for $t$,
  $t'\in\T(U)$.

  If $\lambda'\in\LL_\sigma(U)$ is another $\VR_X(U)$-base of
  $\LL_\sigma(U)$ and $\mu'\in\D(U)$ is such that
  $\Trd(\lambda'\mu')=1$, then $\lambda'=\alpha\lambda$ for some
  $\alpha\in\VR_X(U)^\times$, and the same arguments as in the proof
  of Lemma~\ref{lem:tauto} show that ${\lambda\otimes t\mu\sigma(t') =
  \lambda'\otimes t\mu'\sigma(t')}$ for $t$, $t'\in\T(U)$. Therefore,
  the map $\can_U$ does not depend on the choice of $\lambda$,
  $\mu$. Gluing the maps $\can_U$ for the subschemes $U$ in an open
  cover of $X$ yields a morphism of $\VR_X$-modules $\can_X\colon
  \T\otimes_{\VR_X}\sigma(\T) \to \LL_\sigma\otimes_{\VR_X}\D$.

  On the other hand, it is clear that for every affine open subscheme
  $U\subset X$,  the multiplication in $\D(U)$ yields a map
  $\sigma(\T)(U)\otimes_{\D(U)}\T(U)\to\LL_\sigma(U)$ since
  $L_\sigma=\mult(\sigma(T)\otimes_{D_F}T)$. Therefore, there is a
  morphism of $\VR_X$-modules $\mult_X\colon
  \sigma(\T)\otimes_{\D}\T \to\LL_\sigma$.

  The morphisms $\can_X$ and $\mult_X$ are injective since their
  restrictions to the generic fiber are injective by
  Lemma~\ref{lem:tauto}. Therefore, it only remains to prove that the
  maps induced by $\can_X$ and $\mult_X$ are surjective on the stalks
  at each point, or (by Nakayama's lemma) on the fibers at each
  point. For each point $\pspot$ of $X$, the fiber $\Tatp$ of $\T$
  at $\pspot$ is a left ideal of dimension~$n$ in $\Datp$ since
  $\T$ is a locally free $\VR_X$-module of rank~$n$, and the fiber
  $\Lsigatp$ of $\LL_\sigma$ is
  $\Tatp\cap\sigma(\Tatp)$. Lemma~\ref{lem:tauto} with
  $A=\Datp$ and $I=\Tatp$ shows that the maps $\mult$ and $\can$
  yield isomorphisms
  \[
    \sigma(\Tatp)\otimes_{\Datp} \Tatp \overset{\lowsim}{\lra}
    \Lsigatp \quad\text{and}\quad
    \Tatp\otimes_{\resatp}\sigma(\Tatp) \overset{\lowsim}{\lra}
    \Lsigatp\otimes_{\resatp} \Datp.
  \]
  The proof is thus complete.
\end{proof}

\section{Symmetric spaces over Severi--Brauer varieties}
\label{sec:SBvar}

We use the same notation as in the preceding section: $X$ is the
Severi--Brauer variety of the division algebra $D$ with symplectic
involution $\sigma$, and $\LL_\sigma=\T\cap\sigma(\T)$ is the
invertible $\VR_X$-module obtained by intersecting the
sheaf $\T$ defined in~\eqref{eq:defT}
and its conjugate $\sigma(\T)$. Throughout this section, unadorned
tensor products and $\calHom$ of $\VR_X$-modules are over $\VR_X$.
\medskip

Our goal in this section is to define a canonical isomorphism between
the Witt groups $W^-(D,\sigma)$ and $W(X,\LL_\sigma)$;  see
Theorem~\ref{thm:M}. The construction involves the Witt group
$W^-(\D,\sigma)$, which is shown in Proposition~\ref{prop:Moritath} to be
canonically isomorphic to $W(X,\LL_\sigma)$.

The Witt groups $W(X,\LL_\sigma)$ and $W^-(\D,\sigma)$ are obtained by
Knebusch's construction (see~\cite[Definition~27]{Bal}) from categories with
duality: Let $\VB_X$ be the category of locally free $\VR_X$-modules
of finite rank and $\Mod_\D$ the category of locally free
$\VR_X$-modules with an action of $D$ on the right, in other words, 
right $\D$-modules that are locally free of finite rank as
$\VR_X$-modules. Tensoring with $\T$ (resp.\ $\sigma(\T)$) yields
functors
\begin{alignat*}{2}
  \Theta\colon& \Mod_\D\lra \VB_X,\quad& \V&\longmapsto \V\otimes_\D\T,\\
  \Psi\colon& \VB_X\lra\Mod_\D,\quad& \M&\longmapsto \M\otimes\sigma(\T).
\end{alignat*}
Since $\T\otimes\sigma(\T)$ is an invertible $\D$-bimodule and
$\sigma(\T)\otimes_\D\T$ is an invertible $\VR_X$-module (see
Proposition~\ref{prop:newL}), $\Psi\circ\Theta$ and $\Theta\circ\Psi$
are naturally equivalent to the identity on $\Mod_\D$ and $\VB_X$, 
respectively;  hence $\Theta$ and $\Psi$ are equivalences of
categories.

By definition, $W(X,\LL_\sigma)=W(\VB_X,*,\varpi)$, where $(*,\varpi)$
is the duality defined on $\VB_X$ by $\M^*=\calHom(\M,\LL_\sigma)$ for
every locally free $\VR_X$-module of finite rank $\M$, with
$\varpi_\M\colon \M\xrightarrow{\lowsim}\M^{**}$ the usual
identification. 

On the other hand, a duality $(\sharp,\pi)$ is defined on $\Mod_\D$ by
\[
  \V^\sharp=\calHom_\D(\V,\D) \quad\text{for every object $\V$ in
    $\Mod_\D$}
\]
(where the left $\D$-module structure on $\calHom_\D(\V,\D)$ is
twisted by $\sigma$ into a right $\D$-module structure), and
$\pi_\V\colon \V\xrightarrow{\lowsim}\V^{\sharp\sharp}$ is the usual
identification. By definition,
\[
  W(\D,\sigma)=W(\Mod_\D,\sharp,\pi) \quad\text{and}\quad
  W^-(\D,\sigma)=W(\Mod_\D, \sharp,-\pi).
\]

In order to obtain a canonical isomorphism $W^-(\D,\sigma)\to
W(X,\LL_\sigma)$, we first define a natural transformation
$\theta\colon \Theta\circ\sharp\to *\circ\; \Theta$ and then deduce a
morphism of categories with duality
\[
  (\Theta,\theta)\colon (\Mod_\D,\sharp,-\pi) \lra (\VB_X,*,\varpi).
\]

\begin{lem}
  \label{lem:newcaniso}
  For every object $\V$ in $\Mod_\D$, there is a canonical isomorphism
  of $\VR_X$-modules
  \[
    \theta_\V\colon \V^\sharp\otimes_\D\T \overset{\lowsim}{\lra}
    (\V\otimes_\D\T)^*.
  \]
  It is determined on the stalks at any point
  $\pspot$ by 
   \[
    \left\langle\theta_{\V_\pspot}\left(x^\sharp\otimes t\right),\, y\otimes
    t'\right\rangle_{\LL_\sigma} 
    =
    \mult_\pspot\left(\sigma(t)\left\langle x^\sharp,y\right\rangle_\D\otimes t'\right)
    \quad\text{for $x^\sharp\in\V_\pspot^\sharp$, $y\in \V_\pspot$, 
      and $t$, $t'\in\T_\pspot$,}
  \]
  where 
  $\langle\text{\textvisiblespace}\,,\text{\textvisiblespace}\rangle_\D
  \colon \V_\pspot^\sharp\times \V_\pspot\to \D$ and
  $\langle\text{\textvisiblespace}\,,
  \text{\textvisiblespace}\rangle_{\LL_\sigma} 
  \colon
  (\V_\pspot\otimes_\D\T_\pspot)^*\times(\V_\pspot\otimes_\D\T_\pspot)\to
  \LL_\sigma$ are the canonical bilinear maps. 
\end{lem}

\begin{proof}
  The switch map is an isomorphism
  $\V^\sharp\otimes_\D\T\xrightarrow{\lowsim}
  \sigma(\T)\otimes_\D\calHom_\D(\V,\D)$. We first prove that the
  latter tensor product is isomorphic to
  $\calHom_\D\left(\V,\sigma(\T)\right)$. 

  Let $R$ be an arbitrary commutative $k$-algebra. Let
  $D_R=D\otimes_kR$, and let $M$ be a right $D_R$-module. Write
  $M_0$ for the right $R$-module obtained from $M$ by forgetting the
  $D$-action. By~\cite[Proposition~2.1]{MT}, $M$ is a direct summand of the
  right $D_R$-module $M_0\otimes_kD$. We recall the argument for the
  convenience of the reader: The Goldman element $g=\sum a_i\otimes
  b_i\in D\otimes_kD$ is defined by the condition that $\sum
  a_ixb_i=\Trd_D(x)$ for all $x\in D$; see \cite[Equation~(3.5)]{BoI}.
  It satisfies the property that $(a\otimes b)g=g(b\otimes a)$ for
  every $a$, $b\in D$; see~\cite[Equation~(3.6)]{BoI}. If $u\in
  D$ is such that $\Trd_D(u)=1$, then $\sum aa_iu\otimes b_i = \sum
  a_iu\otimes b_ia$ for all $a\in D$, and the map $M\to M_0\otimes_kD$
  which carries $m\in M$ to $\sum(ma_iu)\otimes b_i$ is an injective
  $D_R$-module homomorphism split by the multiplication map
  $M_0\otimes_kD\to M$.

  If $M_0$ is a projective $R$-module, then $M_0\otimes_kD$ is a
  projective $D_R$-module, so $M$ also is a projective
  $D_R$-module. For every $D_R$-module $N$, the canonical homomorphism
  \[
  N\otimes_{D_R}\Hom_{D_R}(M,D_R) \lra \Hom_{D_R}(M,N)
  \]
  is then an isomorphism; see \cite[Section~II.4.2, p.~II.75]{Bou}.
  This applies in
  particular with $M$ and $N$ the modules of sections of $\V$
  and $\sigma(\T)$ over any affine open subscheme of $X$, or the
  stalks of 
  $\V$ and $\sigma(\T)$ at any point of $X$, and yields an
  isomorphism 
  \[
    \tau_\V\colon\V^\sharp\otimes_\D\T \overset{\lowsim}{\lra}
    \calHom_\D\left(\V,\sigma(\T)\right),\]
    which is given on the stalk at any point $\pspot$ by 
   \[ \tau_{\V_\pspot}(x^\sharp\otimes t)\colon y\longmapsto \sigma(t)
    \langle x^\sharp, y\rangle_\D
    \quad\text{for $x^\sharp\in\V_\pspot^\sharp$, $y\in \V_\pspot$, and
      $t\in\T_\pspot$.}
  \]

  The isomorphism $\theta_\V$ is obtained by composing $\tau_\V$ with
  the isomorphisms
  \[
    \calHom_\D\left(\V,\sigma(\T)\right) \overset{\lowsim}{\lra}
    \calHom(\V\otimes_\D\T,\sigma(\T)\otimes_\D\T) \overset{\lowsim}{\lra}
    \calHom(\V\otimes_\D\T,\LL_\sigma)
  \]
  that arise from the equivalence of categories $\Theta$ and the
  isomorphism $\mult_X$ of Proposition~\ref{prop:newL}.
\end{proof}

The isomorphisms $\theta_\V$ of Lemma~\ref{lem:newcaniso} define a
natural transformation $\theta\colon \Theta\circ\,\sharp
\xrightarrow{\lowsim} *\circ\,\Theta$. We next show that the pair
$(\Theta,\theta)$ is a morphism of categories with duality as
in~\cite[Definition~5]{Bal} (a ``duality-preserving functor'' in the
terminology of~\cite[Section~II(2.6)]{Knus}).

\begin{prop}
  \label{prop:Moritath}
  The pair $(\Theta,\theta)$ induces an isomorphism of Witt
  groups
  \[
    \Mor\colon
    W^-(\D,\sigma)\lra W(X,\LL_\sigma)
  \]
  by mapping the Witt class of every skew-hermitian space
    $(\V,\varphi)$ over $(\D,\sigma)$ to the Witt class of the
    symmetric bilinear space 
    $(\V\otimes_\D\T,\theta_\V\circ(\varphi\otimes\Id_\T))$.
  \end{prop}

\begin{proof}
  To see that $(\Theta,\theta)$ is a morphism of categories with
  duality, it remains to prove that the following diagram commutes for
  every object $\V$ in $\Mod_\D$:
  \begin{equation}
    \label{eq:diagVT}
    \begin{split}\xymatrixcolsep{1cm}
      \xymatrix{\V\otimes_\D\T \ar[r]^-{-\pi_\V\otimes\Id_\T}
        \ar[d]_{\varpi_{\V\otimes_\D\T}} &
          \V^{\sharp\sharp}\otimes_\D\T \ar[d]^{\theta_{\V^\sharp}} \\
          (\V\otimes_\D\T)^{**} \ar[r]^{\theta_\V^*} &
          (\V^\sharp\otimes_\D\T)^*\rlap{.}
          }
    \end{split}
  \end{equation}
  We compute on the stalks at any point $\pspot$: For $x\in\V_\pspot$,
  $y^\sharp\in \V_\pspot^\sharp$, and $t$, $t'\in\T_\pspot$,
  \[
    \left\langle \theta_{\V_\pspot^\sharp}\circ\pi_\V(x\otimes t),
    y^\sharp\otimes t'\right\rangle_{\LL_\sigma} =
    \mult_\pspot\left(\sigma(t)\left\langle \pi_\V(x),y^\sharp\right\rangle_\D \otimes t'\right)
    = \mult_\pspot\left(\sigma(t)\sigma\left(\left\langle y^\sharp,x\right\rangle_\D\right)    \otimes t'\right).
  \]
  On the other hand,
  \begin{multline*}
    \left\langle \theta_\V^*\circ\varpi_{\V\otimes_\D\T}(x\otimes t),
    y^\sharp\otimes t'\right\rangle_{\LL_\sigma} =
    \left\langle \varpi_{\V\otimes_\D\T}(x\otimes t),
    \theta_\V\left(y^\sharp\otimes t'\right)\right\rangle_{\LL_\sigma}
    \\
    =\left\langle \theta_\V\left(y^\sharp\otimes t'\right),x\otimes
    t\right\rangle_{\LL_\sigma} = \mult_\pspot\left(\sigma(t')\left\langle y^\sharp,
    x\right\rangle_\D \otimes t\right).
  \end{multline*}
  Since $\LL_\sigma\subset\Skew(\sigma)$, it follows that
  $\mult_\pspot(\sigma(t_1)\otimes t_2) = -
  \mult_\pspot(\sigma(t_2)\otimes t_1)$ for all $t_1$,
  $t_2\in\T_\pspot$; hence the computation above yields
  \[
    \theta_\V^*\circ\varpi_{\V\otimes_\D\T}(x\otimes t) =
    -\theta_{\V_\pspot^\sharp} \circ\pi_\V(x\otimes t)
    \quad\text{for all $x\in\V_\pspot$, $t\in\T_\pspot$}.
  \]
  Therefore, the diagram~\eqref{eq:diagVT} commutes, and
  $(\Theta,\theta)$ is a morphism of categories with duality. The
  induced homomorphism of Witt groups $\Mor\colon W^-(\D,\sigma) \to
  W(X,\LL_\sigma)$ is an isomorphism because $\Theta$ is an
  equivalence of categories.
\end{proof}

The main theorem of this section follows.

\begin{thm}
  \label{thm:M}
  The composition of the scalar extension map $\ext_X\colon
  W^-(D,\sigma)\to W^-(\D,\sigma)$ 
  with the map $\Mor\colon W^-(\D,\sigma)\xrightarrow{\lowsim} W(X,\LL_\sigma)$ is
  an isomorphism
  \[
    M\colon W^-(D,\sigma)\overset{\lowsim}{\lra} W(X,\LL_\sigma).
  \]
\end{thm}

\begin{proof}
  We first show that $\ext_X$ is injective. By a theorem of Karpenko
  \cite{Kar}, the scalar extension map $\ext_F\colon W^-(D,\sigma) \to
  W^-(D_F,\sigma)$ is injective. The injectivity of $\ext_X$ then follows
  from the commutativity of the following diagram, where $\res_F$ is
  the restriction to the generic fiber:
  \[
  \xymatrix{W^-(D,\sigma) \ar[rr]^{\ext_F}\ar[dr]_{\ext_X} & &
    W^-(D_F,\sigma)\\ 
    &W^-(\D,\sigma)\rlap{.}\ar[ur]_{\res_F}&}
\]

  As $\Mor$ is bijective, to complete the proof, it suffices to show
  $M$ is onto. For this, we use 
  Pumpl\"un's results in \cite{Pump2} (keeping in mind that Pumpl\"un
  chooses as a generator for $\Pic(X)$ an invertible
    $\VR_X$-module
  isomorphic to 
  $\calHom(\LL_\sigma,\VR_X)$ instead of $\LL_\sigma$).

  According to \cite[Theorem~4.3]{Pump2}, $W(X,\LL_\sigma)$ is generated
  by the Witt classes of symmetric bilinear spaces with underlying
 $\VR_X$-module $\tr_{\ell/k}(\N)$, where $\ell$
  is a maximal separable subfield of $D$ and $\N$ is a self-dual
  invertible $\VR_{X_\ell}$-module. Since by
  Proposition~\ref{prop:newL},  $(\LL_\sigma)_\ell\simeq\VR_{X_\ell}(-2)$,
  self-dual invertible $\VR_{X_\ell}$-modules~$\N$ for the
  duality~$*$ are 
  isomorphic to $\VR_{X_\ell}(-1)$; hence $\tr_{\ell/k}(\N)\simeq\T$
  by Proposition~\ref{prop:newT}. We can compare every isomorphism
  $\varphi\colon \T\to \T^*$ to the canonical isomorphism
  ${\theta_\D\circ \Theta(\sigma)}\colon\T\to\T^*$, viewing
  $\sigma$ as an 
  isomorphism $\D\to\D^\sharp$. Since $\End\T=D$ by
  Proposition~\ref{prop:newT}, for every $\varphi$, there exists a 
  $d\in D^\times$ such that the following diagram commutes:
  \[
    \xymatrix{\T\ar[drr]^{\varphi} \ar[d]_{d\cdot} & &\\
      \T\ar[rr]_{\theta_\D\circ\Theta(\sigma)}&&\T^*\rlap{.}}
  \]
  On each stalk $\T_\pspot$, the canonical isomorphism
    $\theta_\D\circ\Theta(\sigma)$ maps
    $t\in\T_\pspot$ to the linear map 
    $\T_\pspot\to (\LL_\sigma)_\pspot$ that carries $t'$ to
    $\sigma(t)t'$; hence $\varphi(t)$ maps
    $t'$ to 
    $\sigma(dt)t'$. The element $d$ satisfies
    $\sigma(d)=-d$ since 
    $\varphi$ is skew-hermitian; hence the Witt class of
    $(\T,\varphi)$ is the image under $M$ of the Witt class of the
    skew-hermitian form $\qf{-d}$ over $(D,\sigma)$. Therefore, the
    map $M$ is onto.
\end{proof}

\section{An octagon of Witt groups}
\label{sec:octa}

Henceforth, we assume $D$ is a quaternion division algebra;  hence
$\sigma$ is the canonical conjugation involution~$\invo$. We write
simply $W^+(D)$ (resp.\ $W^-(D)$) for the Witt group of hermitian
(resp.\ skew-hermitian) forms over~$D$.

Let $i$, $j\in D$ be nonzero anticommuting quaternions, and let
$K=k(i)\subset D$. We have $D=K\oplus jK$; 
hence for every $\varepsilon$-hermitian form $h\colon V\times V\to D$
(with $\varepsilon=\pm 1$) on a right $D$-vector space $V$,  we may
define an  
$\varepsilon$-hermitian form $f\colon V\times V\to K$ (for the
nontrivial automorphism $\invo$ on $K$) and a
$(-\varepsilon)$-symmetric bilinear form $g\colon V\times V\to K$ by
the equation 
\[
h(v,v')=f(v,v')+j g(v,v')\quad\text{for $v$, $v'\in V$.}
\]
We thus obtain Witt group homomorphisms
\[
\pi_1\colon W^\varepsilon(D)\lra W^\varepsilon(K,\invo),\quad h\longmapsto f
\quad\text{and}\quad 
\pi_2\colon W^\varepsilon(D)\lra W^{-\varepsilon}(K),\quad h\longmapsto g; 
\]
see \cite[Lemma~10.3.1]{Sch}.\footnote{There are several typos on
  p.~359 of \cite{Sch}.} 
Computation yields 
an explicit description of $\pi_2\colon
W^-(D) \to W(K)$ ($=W^+(K)$): For 
$h=\qf{q}$, with 
$q=iq_0+jq_1$, $q_0\in k$,  and
$q_1\in K$, we have
\begin{equation}
  \label{eq:defpi2}
  \pi_2(\qf q)=
  \begin{cases}
    \qf{q_1} \qf{1,-q^2} & \text{if $q_1\neq0$,}\\
    0 & \text{if $q_1=0$}.
  \end{cases}
\end{equation}
Indeed, for all $\lambda\in K$ and $\mu\in K$, we have $g(\lambda+j\mu)=q_1\lambda^2-2iq_0\lambda\mu-b\bar q_1\mu^2$. Hence this quadratic form represents $q_1$ and has discriminant $q^2$.

We may also define maps in the opposite direction
using scaled base change. More precisely, for every
$\varepsilon$-hermitian form 
$f\colon U\times U\to K$ on a $K$-vector space $U$,  
there is a unique $(-\varepsilon)$-hermitian form
\[
h\colon (U\otimes_KD)\times(U\otimes_KD)\lra D
\]
such that $h(u,u')=f(u,u')i$ for $u$, $u'\in U$. 
Similarly, for every $\varepsilon$-symmetric bilinear form $g\colon
U\times U\to K$ on a $K$-vector space $U$,  there is a unique
$(-\varepsilon)$-hermitian form
\[
h'\colon(U\otimes_KD)\times(U\otimes_KD)\lra D
\]
such that $h'(u,u')=ijg(u,u')$ for $u$, $u'\in U$. 
Thus, for $\varepsilon=\pm1$, we obtain Witt group homomorphisms 
\[
\sigma_1\colon W^\varepsilon(K,\invo)\lra W^{-\varepsilon}(D),\quad f\longmapsto h \quad\text{and}\quad
\sigma_2\colon W^\varepsilon(K)\lra W^{-\varepsilon}(D),\quad g\longmapsto h'.
\]
(Of course, $W^-(K)=0$.)

\begin{thm}
  \label{thm:octa}
  The following octagon is exact:
  \[
  \xymatrix{
  W^+(D)\ar[r]^{\pi_2} & W^{-}(K) \ar[r]^{\sigma_2} &
  W^+(D)\ar[d]^{\pi_1}\\
  W^-(K,\invo)\ar[u]^{\sigma_1}& &W^+(K,\invo)\ar[d]^{\sigma_1}\\
  W^-(D)\ar[u]^{\pi_1}&W^+(K) \ar[l]_{\sigma_2}& W^-(D)\rlap{.}\ar[l]_{\pi_2}
  }
  \]
\end{thm}

\begin{proof}
  The exactness of the five-term sequence from $W^-(K)$ to $W^+(K)$ is
  proved in \cite[Theorem~10.3.2]{Sch}. The same arguments can be used to
  prove the exactness of the other half; see also \cite[Proposition~2]{Lewis} or
  \cite[Section~6]{GBM}. 
\end{proof}

From here on, we omit the superscripts $+$. To define the first map
in~\eqref{eq:intro2} and the last map in~\eqref{eq:intro3},
note that every hermitian form on $D$ has a diagonalization with
coefficients in $k$;  hence scalar extension yields a surjective group
homomorphism 
\[
\ext_D\colon W(k)\lra W(D).
\]
On the other hand, for every hermitian form $h$ on $D$, the map
$q_h\colon v\mapsto h(v,v)$ is a quadratic form on $k$, and mapping
$h$ to $q_h$ yields a group homomorphism
\[
s_D\colon W(D)\lra W(k).
\]
The following result is proved in~\cite[Theorem~10.1.7]{Sch}. 

\begin{lem}
  \label{lem:quat}
  The following diagram, where $n_D$ denotes multiplication by the
  norm form of $D$, commutes and has exact row and column:
  \[
  \xymatrix{
  &W(k)\ar[d]_{\ext_D}\ar[dr]^{n_D}& \\
  0\ar[r]& W(D)\ar[r]^{s_D}\ar[d] & W(k)\\
  &0\rlap{.}&
  }
  \]
\end{lem}

\section{Residues and transfers}
\label{sec:restrans}

Recall that $D$ is now assumed to be a quaternion division algebra; 
hence its Severi--Brauer variety $X$ is a smooth projective conic.
The maps in the exact sequences \eqref{eq:intro2} and
\eqref{eq:intro3} depend on the choice of uniformizers $\pi_\pspot$
and linear functionals $s_\pspot$ at each closed point $\pspot\in
X^{(1)}$. In order to make suitable choices, we first introduce
coordinates, which will allow us to write an equation for $X$ and to
identify the invertible $\VR_X$-module $\LL_\sigma$ with the
ideal sheaf of a point 
$\infty\in X^{(1)}$. Since there is a unique symplectic involution
$\sigma$ on~$D$, namely, the conjugation involution $\invo$, we
simplify the notation by writing $\LL$ for $\LL_\sigma$.

\subsection{Coordinatization}
\label{subsec:coord}

Let $D^0$ be the $3$-dimensional
$k$-vector space of pure quaternions in $D$. We have $d^2\in k$ for
every $d\in D^0$; hence the map $q\colon D^0\to k$ defined by
$q(d)=d^2$ is a quadratic form. Since every $2$-dimensional
left ideal of~$D$ over a splitting field intersects $D^0$ in a line
(see Lemma~\ref{lem:tauto}), we may identify $X$ with the conic in the
projective plane $\mathbb{P}(D^0)$ given by the equation $q=0$.
Under this identification, every line in $D^0$ corresponds to
  the left ideal of $D$ that it generates.

Let $i$, $j\in D$ be two nonzero anticommuting pure quaternions, and
let 
$i^2=a$, $j^2=b$, so that
\[
D=(a,b)_k.
\]
The elements $aj$, $ij$, $bi$ form a $k$-base of $D^0$. If $\xi$,
$\eta$, $\zeta$ denotes the dual base, the conic $X$ is given by the
equation 
\[
(aj\xi+ij\eta+bi\zeta)^2=0,\quad\textit{i.e.,}\quad
a\xi^2-\eta^2+b\zeta^2=0.
\]
Let $\infty\in X^{(1)}$ be the closed point given by the equation
$\zeta=0$; the residue field $\resatinf$ at $\infty$ is canonically
isomorphic to $k(i)$ by a map that carries the value $\frac\eta\xi(\infty)$ of the function $\frac \eta\xi$ at $\infty$ to
$i$; see \cite[Proposition~45.12]{EKM}. Let
also $X_\af\subset X$ 
be the open subscheme defined by
  $\zeta\neq0$,
which is an affine conic, and let 
$\VR_\af$ be the affine ring of $X_\af$; then, writing
$x=\frac\xi\zeta$ 
and $y=\frac\eta\zeta$, we have
\[
\VR_\af=k[x, y] \subset k(x,y) = F \quad\text{with $y^2=ax^2+b$.}
\]
The point with coordinates $(x,y)$ on the affine conic
  $X_\af(F)$ is $ajx+ijy+bi$; under the identification of the
  conic with the Severi--Brauer variety of $D$, it is the $F$-rational
  point obtained by base change from the generic point $T$ of
  $X$. 
  Therefore, the generic fiber of the sheaf $\T$ defined
  in~\eqref{eq:defT} is  
\[
T=D_Fe,\quad\text{where }
e=bi+axj+yij\in D_F.
\]
Clearly, $e\in T\cap\overline T$ since $\overline e = -e$;  hence the
generic fiber $T\cap\overline T$ of $\LL$ is $L=eF$. We
next describe the module of affine sections and the stalk at $\infty$
of $\T$ and $\LL$, for which we use the notation $\T_\af$,
$\T_\infty$, $\LL_\af$, $\LL_\infty$. We write
$\D_\af=D\otimes\VR_\af$ for the module of sections of $\D$ 
over $X_\af$ and $\D_\infty=D\otimes\VR_\infty$ for the stalk of $\D$
at $\infty$.

\begin{prop}
  \label{prop:afsec}
  We have $T=eF+jeF=jeF+ijeF$ and
  \[
  \T_\af=\D_\af\, e,\quad\T_\infty=\D_\infty ex^{-1},\quad
  \LL_\af=e\VR_\af,\quad\LL_\infty=ex^{-1}\VR_\infty.
  \]
\end{prop}

\begin{proof}
  Because $e^2=0$, we have
  \begin{equation}
  \label{eq:esq0}
  bie+axje+yije=0.
  \end{equation}
  Multiplying on the left by $i^{-1}$, we obtain
  \[
  be+xije+yje=0. 
  \]
  Since $T=D_F e=eF+ieF+jeF+ijeF$, the first assertion follows from
  these equations. 

  By definition,
  \[
  \T_\af= T\cap \D_\af =(D_Fe)\cap\D_\af\quad\text{and}\quad
  \LL_\af=L\cap 
  \D_\af=(eF)\cap\D_\af.
  \]
  Since $e\in\D_\af$, the inclusions $\D_\af e\subset\T_\af$ and
  $e\VR_\af\subset \LL_\af$ are clear. If $\lambda\in F$ is such that
  $e\lambda\in\D_\af$, then by looking at the coefficient of $i$ in
  $e\lambda$,  we see that $\lambda\in\VR_\af$. Therefore,
  $\LL_\af=e\VR_\af$. Now, the first part of the proof shows that
  every element in $T$ can be written as $je\lambda +ije\mu$ for some
  $\lambda$, $\mu\in F$. If $\lambda$, $\mu$ are such
  that $je\lambda+ije\mu\in\D_\af$, then inspection of the
  coefficients of $j$ and $ij$ shows that $\lambda$,
  $\mu\in\VR_\af$. Therefore, $je\lambda+ije\mu\in \D_\af e$, and it
  follows that $\T_\af=\D_\af e$.

  Next, we consider the stalks at $\infty$. By definition,
  \[
  \T_\infty=T\cap\D_\infty=(D_F e)\cap\D_\infty\quad\text{and}\quad
  \LL_\infty=L\cap 
  \D_\infty = (eF)\cap\D_\infty.
  \]
  Since $ex^{-1}\in\D_\infty$, we have $ex^{-1}\in\T_\infty$ and
  $ex^{-1}\in\LL_\infty$, so the inclusions
  $\D_\infty ex^{-1}\subset\T_\infty$ and 
  $ex^{-1}\VR_\infty \subset\LL_\infty$ are clear. If $\lambda\in F$ is
  such that 
  $e\lambda\in \D_\infty$, then the coefficient of $j$
  shows that $x\lambda\in\VR_\infty$, hence
  $e\lambda\in\, ex^{-1}\VR_\infty$. To complete the description of 
  $\T_\infty$, we use the equation $T=eF+jeF$ proven above. If 
  $e\lambda +je\mu\in\D_\infty$
for some $\lambda$, $\mu\in F$, then by
  looking at the coefficients of $1$ and $j$, we see that $x\lambda$,
  $x\mu\in\VR_\infty$. Therefore, $e\lambda +je\mu\in \D_\infty
  ex^{-1}$. So we get $\LL_\infty=ex^{-1}\VR_\infty$ and
  $\T_\infty=\D_\infty ex^{-1}$.  
\end{proof}

From the descriptions of $\LL_\af$ and $\LL_\infty$ above, it follows
that mapping $e\in L$ to $1\in F$ defines an isomorphism of
invertible $\VR_X$-modules
\begin{equation}
  \label{eq:epsilon}
  \varepsilon\colon\LL\overset{\lowsim}{\lra} \II(\infty),
\end{equation}
where $\II(\infty)$ is the ideal sheaf of $\infty$, \textit{i.e.}, the subsheaf
of $\VR_{X}$ whose module of affine sections is $\VR_\af$ and whose
stalk at $\infty$ is the maximal ideal
$\mathfrak{m}_\infty=x^{-1}\VR_\infty$ of $\VR_\infty$. 

\subsection{Coherent choices}
\label{subseq:coherent}

For each point $\pspot\in X^{(1)}$, let $\pi_\pspot$ be a uniformizer
of the local ring $\VR_\pspot$. Two residue maps
\[
  \partial_\pspot^1,\;\partial_\pspot^2\colon W(F)\lra W(\resatp)
\]
are defined as follows: Select
in every symmetric bilinear space $(W,b)$
over $F$ an $\VR_\pspot$-lattice of the form
$W_1\perp W_2$, such that the restrictions $b_1$ of $b$ to
$W_1$ and $b_2$ of $\qf{\pi_\pspot^{-1}}b$ to $W_2$ are nonsingular (as $\VR_\pspot$-bilinear
forms). Then $\partial_\pspot^i$ maps $(W,b)$ to the Witt class of
$(W_i\otimes_{\VR_\pspot}\resatp,(b_i)_{\resatp})$. Thus,
$\partial_\pspot^1$ does not depend on the choice of $\pi_\pspot$, but
$\partial_\pspot^2$ does. If the bilinear space $(W,b)$ is the generic
fiber of a symmetric bilinear space $(\W,b)$ over $X$ with
values in 
$\VR_{X}$, then by definition\footnote{We abuse notation by not
  distinguishing between a space and its Witt class.}
\[
\partial_\pspot^1(W,b) = \left(\W_\pspot\otimes_{\VR_\pspot}\resatp,
b_{\resatp}\right) \quad\text{and}\quad
\partial_\pspot^2(W,b)=0 \quad\text{for all $\pspot\in X^{(1)}$.}
\]
In contrast, if $(W,b)$ is the generic fiber of a symmetric bilinear
space $(\W,b)$ with values in $\II(\infty)$, then
\begin{align*}
   \partial_\pspot^1(W,b) = &
  \begin{cases}
    (\W_\pspot\otimes_{\VR_\pspot}\resatp,b_{\resatp}) &\text{if
      $\pspot\neq\infty$},\\ 
    0&\text{if $\pspot=\infty$},
  \end{cases}
  \\
  \intertext{and there exists an $\alpha\in \resatinf^\times$ depending on
  the choice of $\pi_\infty$ such that}
  \partial_\pspot^2(W,b) = &
  \begin{cases}
    0&\text{if $\pspot\neq\infty$},\\
    (\W_\infty\otimes_{\VR_\infty}\resatinf,\qf\alpha b_{\resatinf})
    &\text{if $\pspot=\infty$}.
  \end{cases}
\end{align*}
It follows that the maps
\[
\delta=\oplus_{\pspot}\partial_\pspot^2\colon W(F) \lra
\bigoplus_{\pspot\in X^{(1)}} W(\resatp)
\quad\text{and}\quad
\delta'=\left(\oplus_{\pspot\neq\infty}\partial_\pspot^2\right)
\oplus \partial^1_\infty \colon W(F) \lra
\bigoplus_{\pspot\in X^{(1)}} W(\resatp)
\]
vanish on the images of $W(X)$ and $W(X,\II(\infty))$, 
respectively. These maps fit in the exact
sequences~\eqref{eq:intro1}. 

To extend these exact sequences further, we use transfer maps
$W(\resatp)\to W(k)$.  
Since $\delta$ and $\delta'$ depend on the choice of $\pi_\pspot$, we need to make a coherent choice for these transfers. With this in mind,  
we consider the Weil differential
$\omega=\frac{dx}{2y}$, which is uniquely determined up to a factor in
$k^\times$ by the condition that its divisor is $-\infty$ (see \cite[Section~II.5]{Chev}). Thus, for $\pspot\in X^{(1)}$,  the $\pspot$-component
$\omega_\pspot$ is 
a linear map $F\to k$ that vanishes on $\VR_\pspot$ if
$\pspot\neq\infty$ and on 
$\mathfrak{m}_\infty$ if $\pspot=\infty$. Abusing notation, we again write
$\omega_\pspot$ for the following $k$-linear maps induced by the local
components of $\omega$:
\[
\omega_{\pspot}\colon \mathfrak{m}_{\pspot}^{-1}/\VR_{\pspot}\lra k
\quad\text{for $\pspot\neq\infty$},\quad \text{and}\quad
\omega_\infty\colon\VR_{\infty}/\mathfrak{m}_{\infty}=\resatinf\lra k.
\]
For $\omega=\frac{dx}{2y}$, the computation in \cite[Section~VI.3]{Chev} shows that $\omega_\infty$ is defined by
\begin{equation}
\label{eq:omegadef}
\omega_\infty(1)=0 \quad\text{and}\quad
\omega_\infty\left(\frac{y}x(\infty)\right)=-1,
\end{equation}
where $\frac{y}x(\infty)$ is the image of $x^{-1}y\in\VR_\infty$ in
$\resatinf$. In the following description of the maps
$\omega_{\pspot}$ for $\pspot\neq\infty$, we write $v_\qspot$ for the
(normalized) $\qspot$-adic valuation on $F$, for every $\qspot\in X^{(1)}$.

\begin{prop}
  \label{prop:diffdesc}
  For $\pspot\neq\infty$, every element in
  $\mathfrak{m}_\pspot^{-1}/\VR_\pspot$ can be represented in the form
  $f+\VR_\pspot$ for some $f\in\mathfrak{m}_\pspot^{-1}$ such that
  $v_\qspot(f)\geq0$ for all $\qspot\neq\pspot$. For such $f$,
  \[
  \omega_\pspot(f)=-\omega_\infty\left(f(\infty)\right).
  \]
\end{prop}

\begin{proof}
  As in \cite[Section~II.1]{Chev}, let\footnote{We use Chevalley's
    notation from~\cite{Chev}. Most references use the notation
    $L(\pspot)$ for Chevalley's $\mathfrak{L}(-\pspot)$.}
  \[
  \mathfrak{L}(-\pspot)=\{f\in F\mid v_\pspot(f)\geq-1\text{ and
    $v_\qspot(f)\geq0$ for $\qspot\neq\pspot$}\}.
  \]
  Since $F$ is an algebraic function field of genus zero, it follows
  from the Riemann--Roch theorem that 
  $\mathfrak{L}(-\pspot)$ is a $k$-vector space of
  dimension~$1+\deg\pspot$; see \cite[Corollary, p.~32]{Chev}. Consider the
  $k$-linear map 
  \[
  \varphi\colon\mathfrak{L}(-\pspot)\to
  \mathfrak{m}_\pspot^{-1}/\VR_\pspot \quad\text{given by }
  \varphi(f)=f+\VR_\pspot.
  \]
  Its kernel consists of elements $f\in F$ such that
  $v_\qspot(f)\geq0$ for all $\qspot\in X^{(1)}$, \textit{i.e.},
  $\ker\varphi=k$. Since
  $\dim(\mathfrak{m}_\pspot^{-1}/\VR_\pspot)=\deg\pspot$, dimension
  count shows that $\varphi$ is onto, which proves the first
  statement. For $f\in\mathfrak{L}(-\pspot)$, we have $f\in\VR_\qspot$
  for all $\qspot\neq\pspot$, hence $\omega_\qspot(f)=0$ for all
  $\qspot\neq\pspot$, $\infty$. Since Weil differentials vanish on
  $F$, it follows that
  \[
  \omega_\pspot(f)+\omega_\infty(f)=0,
  \]
  which completes the proof.
\end{proof}

If
$\pi_{\pspot}\in\VR_{\pspot}$ is a uniformizer at $\pspot$, then
$\mathfrak{m}^{-1}_{\pspot}=\pi_{\pspot}^{-1}\VR_{\pspot}$, and
multiplication by $\pi_{\pspot}$ defines an isomorphism of
$\resatp$-vector spaces
\[
\mu_{\pi_{\pspot}}\colon \mathfrak{m}_{\pspot}^{-1}/\VR_{\pspot}
\overset{\lowsim}{\lra} \VR_{\pspot}/\mathfrak{m}_{\pspot}=\resatp.
\]

\begin{defn}
  \label{def:coherent}
  A choice of uniformizer $\pi_\pspot$ and of $k$-linear functional
  $s_\pspot\colon \resatp\to k$ is said to be \emph{coherent} at
  $\pspot\in X^{(1)}_\af$ if the following diagram commutes:
  \[
  \xymatrix{
  \mathfrak{m}^{-1}_\pspot/\VR_\pspot \ar[rr]^{\omega_\pspot}
  \ar[dr]_{\mu_{\pi_\pspot}} & & k\\
  &\resatp\rlap{.}\ar[ur]_{s_\pspot}&
  }
  \]
\end{defn}

\begin{prop}
  \label{prop:coherent}
  Let $\pi_\pspot$, $s_\pspot$ and $\pi'_\pspot$, $s'_\pspot$ be
  coherent choices of uniformizer and linear functional at $\pspot\in
  X^{(1)}_\af$. The corresponding residue maps $\partial_\pspot^2$,
  ${\partial'_\pspot}^2\colon W(F)\to W(\resatp)$ and transfer maps
  $(s_\pspot)_*$, $(s'_\pspot)_*\colon W(\resatp)\to W(k)$ make the
  following diagram commute:
  \[
  \xymatrix{
  W(F)\ar[r]^{\partial_\pspot^2} \ar[d]_{{\partial'_\pspot}^2} &
  W(\resatp) \ar[d]^{(s_\pspot)_*} \\
  W(\resatp) \ar[r]^{(s'_\pspot)_*} & W(k)\rlap{.}
  }
  \]
\end{prop}

\begin{proof}
  Let $u=\pi'_\pspot\pi_\pspot^{-1}\in\VR_\pspot^\times$, and let
  $u(\pspot)$ be the image of $u$ in $\resatp^\times$. Then
  $\partial_\pspot^2(\qf f) =
  \qf{u(\pspot)}{\partial_\pspot'}^2(\qf f )$ for all $f\in
  F^\times$. On the other hand, $\mu_{\pi'_\pspot}=u(\pspot)
  \mu_{\pi_\pspot}$; hence $s_\pspot(g)=s'_\pspot(u(\pspot) g)$
  for all $g\in \resatp$ because $s_\pspot\circ \mu_{\pi_\pspot} =
  s'_\pspot\circ \mu_{\pi'_\pspot}$ as $\pi_\pspot$, $s_\pspot$ and
  $\pi'_\pspot$, $s'_\pspot$ are coherent choices. Therefore, the
  following diagrams commute:
  \[
  \xymatrix{
  W(F)\ar[r]^{\partial_\pspot^2} \ar[d]_{{\partial'_\pspot}^2} &
  W(\resatp)\\
  W(\resatp)\rlap{,}\ar[ur]_{\qf{u(\pspot)}}
  }
  \quad
  \xymatrix{
  & W(\resatp)\ar[dl]_{\qf{u(\pspot)}}
  \ar[d]^{(s_\pspot)_*}\\
  W(\resatp) \ar[r]^{(s'_\pspot)_*}&W(k)\rlap{.}
  }
  \]
  The proposition follows.
\end{proof} 

\subsection{Transfer maps}
\label{subsec:transfer}

Besides the transfer maps $(s_\pspot)_*\colon W(\resatp)\to W(k)$,
which fit in the exact sequence~\eqref{eq:intro3}, we also need
transfer maps $W(\resatp)\to W^-(D)$ to complete the
sequence~\eqref{eq:intro2}. For any nonzero linear functional
$s_\pspot\colon \resatp\to k$, the map
$s_{\Datp}=\Id_D\otimes s_\pspot\colon \Datp\to D$ is $D$-linear
for the left and right $D$-vector space structures on $\Datp$, and
commutes with  quaternion conjugation. Moreover, if $\xi\in
\Datp$ is such that $s_{\Datp}(\overline\xi \eta)=0$ for all
$\eta\in \Datp$, then writing
$\xi=1\otimes\xi_0+i\otimes\xi_1+j\otimes\xi_2+ij\otimes\xi_3$ with
$\xi_0$, \ldots, $\xi_3\in \resatp$, we get
\[
s_{\Datp}(\overline\xi \cdot 1\otimes \zeta) = s_\pspot(\xi_0\zeta) -
is_\pspot(\xi_1\zeta) - js_\pspot(\xi_2\zeta) - ij s_\pspot(\xi_3\zeta)=0
\quad\text{for all $\zeta\in \resatp$},
\]
hence $\xi_0=\cdots=\xi_3=0$. It follows that $s_{\Datp}$
is an involution trace in the sense of \cite[Section~I.7.2, p.~40]{Knus}. It induces a homomorphism of Witt groups (see \cite[Section~I.10.3, p.~62]{Knus})
\[
(s_{\Datp})_*\colon W^-(\Datp)\lra W^-(D).
\]
On the other hand, restricting the canonical isomorphism
$\Mor$ of Proposition~\ref{prop:Moritath} to the fiber at $\pspot$,
we obtain an isomorphism
\[
\Mor_\pspot\colon W^-(\Datp) \overset{\lowsim}{\lra} W(\resatp,
\Latp).
\]
Next, consider the restriction of the isomorphism
$\varepsilon\colon\LL\to \II(\infty)$ of~\eqref{eq:epsilon} to
the 
fiber at $\pspot$. If $\pspot\neq\infty$, the fiber of $\II(\infty)$
at $\pspot$ is $\resatp$;  hence $\varepsilon_\pspot$ yields an
isomorphism
\[
(\varepsilon_\pspot)_*\colon W(\resatp,\Latp) \overset{\lowsim}{\lra}
W(\resatp).
\]
We let $t_\pspot$ denote the composition (which depends on the choice
of the linear functional $s_\pspot$)
\[
t_\pspot=(s_{\Datp})_*\circ \Mor_\pspot^{-1} \circ
(\varepsilon_\pspot)_*^{-1} \colon W(\resatp)\lra W^-(D)
\quad\text{for $\pspot\in X^{(1)}_\af$.}
\] 
The fiber of $\II(\infty)$ at $\infty$ is
$\mathfrak{m}_\infty/\mathfrak{m}_\infty^2$; hence the fiber of
$\varepsilon$ at $\infty$ yields an isomorphism
\[
(\varepsilon_\infty)_*\colon W(\resatinf,\Latinf) \overset{\lowsim}{\lra}
W\left(\resatinf,\mathfrak{m}_\infty/\mathfrak{m}_\infty^2\right).
\]
Choosing a uniformizer
$\pi_\infty$ at $\infty$, we define a $\resatinf$-linear isomorphism
\[
\mu_{\pi_\infty}\colon \mathfrak{m}_\infty/\mathfrak{m}_\infty^2
\overset{\lowsim}{\lra} \resatinf \quad\text{by}\quad
\mu_{\pi_\infty}\left(f+\mathfrak{m}_\infty^2\right)=
\frac f{\pi_\infty}(\infty) \quad\text{for } f\in\mathfrak{m}_\infty,
\]
hence an isomorphism
\[
\left(\mu_{\pi_\infty}\right)_*\colon
W\left(\resatinf,\mathfrak{m}_\infty/\mathfrak{m}_\infty^2\right) \overset{\lowsim}{\lra}
W(\resatinf).
\]
Choosing $s_\infty=\omega_\infty$ (defined in~\eqref{eq:omegadef}), we
mimic the definition of $t_\pspot$ for $\pspot\in X^{(1)}_\af$ and
set
\[
t_\infty=\left(s_{D_\infty}\right)_*\circ \Mor_\infty^{-1} \circ
\left(\varepsilon_\infty\right)_*^{-1}\circ \left(\mu_{\pi_\infty}\right)_*^{-1} \colon
W(\resatinf)\lra W^{-}(D). 
\]
Note that the map $t_\infty$ depends on the choice of uniformizer
$\pi_\infty$ via $\mu_{\pi_\infty}$.

We next give an explicit description of the transfer maps
$t_\pspot$. Recall that for every $\pspot\in X^{(1)}$, we write
$\Datp=D\otimes_k\resatp$ for the fiber of $\D$ at $\pspot$. We
also let $\Tatp=\T_\pspot\otimes_{\VR_\pspot}\resatp$ denote the
fiber of $\T$ at $\pspot$. If $\pspot\neq\infty$, we write $e_\pspot$
for the image of $e=bi+axj+yij$ in $\Tatp$; similarly, we let
$e_\infty$ denote the image of $ex^{-1}$ in $\Tatinf$. Thus, from
Proposition~\ref{prop:afsec} it follows that
\[
  \Tatp= \Datp e_\pspot, \quad\text{hence}\quad
  \overline\Tatp=e_\pspot\Datp\quad\text{for all $\pspot\in
    X^{(1)}$.} 
\]

\begin{prop}
  \label{prop:explicitransfer}
  Let $\pspot\in X^{(1)}$ and $f\in \resatp^\times$. If\,
  $\pspot=\infty$, define $t_\infty$ by selecting $x^{-1}$ as a
  uniformizer. For all $\pspot\in X^{(1)}$, the Witt class of
  $t_\pspot(\qf f)$ is represented by the transfer along
  $s_{\Datp}$ of the skew-hermitian form
  \[
  h_f\colon \overline\Tatp\times \overline\Tatp\lra \Datp
  \quad\text{defined by}\quad h_f(e_\pspot\xi,
    e_\pspot\eta) = f\overline\xi e_\pspot\eta \quad\text{for }
    \xi, \eta\in \Datp.
  \]
\end{prop}

\begin{proof}
  First, suppose  $\pspot\neq\infty$. It suffices to show that the
  skew-hermitian form $h_f$
  satisfies
  $(\varepsilon_\pspot)_*\left(\Mor_\pspot(h_f)\right) = \qf
  f$. By definition,
  \[
  \Mor_\pspot(h_f)\colon \left(\overline\Tatp\otimes_{\Datp}
  \Tatp\right) \times \left(\overline\Tatp\otimes_{\Datp}
  \Tatp\right) \lra \overline\Tatp\cdot\Tatp= \Latp
  \]
  carries $(e_\pspot \xi_1\otimes \xi_2 e_\pspot,
  e_\pspot\eta_1\otimes \eta_2 e_\pspot)$ to $-fe_\pspot
  \overline{\xi_2}\overline{\xi_1} e_\pspot \eta_1\eta_2 e_\pspot$ for
  $\xi_1$, $\xi_2$, $\eta_1$, $\eta_2\in \Datp$. Note that
  $\overline\Tatp\otimes_{\Datp}\Tatp$ is a
  $\resatp$-vector space of dimension~$1$, isomorphic to
  $\resatp$ under the composition
  \[
  \overline\Tatp\otimes_{\Datp}\Tatp \overset{m_\pspot}\lra
  \Latp \overset{\varepsilon_\pspot}\lra \resatp,
  \]
  which maps $e_\pspot\xi_1\otimes\xi_2 e_\pspot$ to
  $\Trd(\xi_1\xi_2e_\pspot)$ (see~\eqref{eq:epsilon}). If $\xi\in
  \Datp$ is such that $\Trd(\xi e_\pspot)=1$, then
  $e_\pspot\xi\otimes e_\pspot$ is a $\resatp$-base of
  $\overline\Tatp\otimes_{\Datp}\Tatp$, and
  \[
  \varepsilon_\pspot\left(\Mor_\pspot\left(h_f\right)\left(e_\pspot\xi \otimes
  e_\pspot, e_\pspot\xi\otimes e_\pspot\right)\right) = -f \Trd\left(\overline\xi
  e_\pspot \xi e_\pspot\right).
  \]
  Since $e_\pspot\xi e_\pspot = e_\pspot(\xi e_\pspot - e_\pspot
  \overline\xi) = e_\pspot\Trd(\xi e_\pspot)=e_\pspot$, we have
  \[
  -f\Trd\left(\overline\xi e_\pspot \xi e_\pspot\right) = -f \Trd\left(\overline\xi  e_\pspot\right) = f\Trd\left(\overline{e_\pspot}\overline\xi\right)=f.
  \]
  The proposition is thus proved for $\pspot\neq\infty$.

  For $\pspot=\infty$, the map $\varepsilon_\infty\colon \Latinf \to
  \mathfrak{m}_\infty/\mathfrak{m}^2_\infty$ carries $e_\infty$ to
  $x^{-1}+\mathfrak{m}_\infty^2$; hence
  $\mu_{x^{-1}}\circ\varepsilon_\infty$ maps $e_\infty$ to $1$. The
  same arguments as in the case where $\pspot\neq\infty$ yield the
  proof. 
\end{proof}

\begin{prop}
  \label{prop:coherent2}
  For $\pspot\in X^{(1)}_\af$, let
  $\pi_\pspot$, $s_\pspot$ and $\pi'_\pspot$, $s'_\pspot$ be
  coherent choices of uniformizer and linear functional, and let
  $\partial_\pspot^2$, 
  ${\partial'_\pspot}^2\colon W(F)\to W(\resatp)$ and 
  $t_\pspot$, $t'_\pspot\colon W(\resatp)\to W^-(D)$ be the
  corresponding residue and transfer maps. Similarly, let $\pi_\infty$,
  $\pi'_\infty$ be uniformizers at 
  $\infty$, and let $\partial^2_\infty$, ${\partial'_\infty}^2\colon
  W(F)\to W(\resatinf)$ and $t_\infty$, $t'_\infty\colon W(\resatinf)\to
  W^{-}(D)$ be the corresponding residue and transfer maps. For every
  $\qspot\in X^{(1)}$, the
  following diagram commutes:
  \[
  \xymatrix{
  W(F)\ar[r]^{\partial_\qspot^2} \ar[d]_{{\partial'_\qspot}^2} &
  W(\resatq) \ar[d]^{t_\qspot} \\
  W(\resatq) \ar[r]^{t'_\qspot} & W^{-}(D)\rlap{.}
  }
  \]
\end{prop}

\begin{proof}
  For $\qspot\neq\infty$, the same arguments as in the proof of
  Proposition~\ref{prop:coherent} yield the proof. For
  $\qspot=\infty$, it is readily verified that
  $(\mu_{\pi_\infty})_*^{-1}\circ\partial_\infty^2 =
  (\mu_{\pi'_\infty})_*^{-1}\circ {\partial'_\infty}^2$. The
  proposition follows.
\end{proof}

\section{Exactness of the sequences}
\label{sec:exact}

In this section, we define the maps in the sequences~\eqref{eq:intro2}
and \eqref{eq:intro3}, and prove their exactness.
We start with the sequence \eqref{eq:intro3}. The isomorphism
$\varepsilon\colon\LL\xrightarrow{\lowsim}\II(\infty)$ (see
\eqref{eq:epsilon}) yields a Witt group
isomorphism
\[
\varepsilon_*\colon W(X,\LL) \overset{\lowsim}{\lra} W(X,\II(\infty)).
\]
Recall the isomorphism $M\colon W^-(D) \xrightarrow{\lowsim}
W(X,\LL)$ of Theorem~\ref{thm:M}.
Let $\rho\colon W^-(D)\to W(F)$ be the composition
\[
  W^-(D)\overset{M}\lra
  W(X,\LL) \overset{\varepsilon_*}\lra W\left(X,\II(\infty)\right)
  \xrightarrow{\res_F} W(F).
\]

\begin{thm}
  \label{thm:exseq1}
  Let $\pi_\pspot$, $s_\pspot$ be a coherent choice of uniformizer and
  linear functional at each point $\pspot\in X^{(1)}_\af$, and let
  $s_\infty=\omega_\infty$. 
  The following sequence is exact:
  \begin{equation}
  \label{eq:exseq1}
  0\lra W^-(D) \overset{\rho}\lra W(F) \overset{\delta'}\lra
  \bigoplus_{\pspot\in X^{(1)}} W(\resatp)\xrightarrow{\sum(s_\pspot)_*}
  W(k) \xrightarrow{\ext_D} W(D)\lra 0.
  \end{equation}
\end{thm}

\begin{proof}
  By the definition of $\rho$, the following diagram commutes:
  \[
  \xymatrix{W^-(D) \ar[rr]^{M} \ar[dr]_{\rho}&& W(X,\LL)
    \ar[dl]^{\res_F\circ\varepsilon_*}\\ 
  &W(F)\rlap{.}&
  }
  \]
The exactness of the following sequence is the purity
    property of $W\left(X,\II(\infty)\right)$ established
    in~\cite[Corollary~10.3]{BW}: 
\begin{equation*}
0\lra W\left(X,\II(\infty)\right)\xrightarrow{\res_F} W(F)
\overset{\delta'}\lra 
\bigoplus_{\pspot\in X^{(1)}} W(\resatp).
\end{equation*}
  Since $M$ and $\varepsilon_*$ are isomorphisms, it follows that the
  sequence~\eqref{eq:exseq1} is exact at $W^-(D)$ and $W(F)$.
  The exactness at $\bigoplus_{\pspot\in X^{(1)}}W(\resatp)$ was proved by
  Parimala~\cite[Theorem~5.1]{Pari}. (It is straightforward to check that
  the particular choice of uniformizers and linear functionals in
  \cite[Section~4]{Pari} is coherent, and Proposition~\ref{prop:coherent}
  shows that the exactness of the sequence does not depend on this
  choice.) In \cite[Theorem~6a]{Pf}, Pfister shows that the image of $\sum
  (s_\pspot)_*$ is the kernel of $n_D$ (multiplication by the norm
  form of~$D$). Therefore, the exactness at $W(k)$ and $W(D)$ follows from
  Lemma~\ref{lem:quat}. 
\end{proof}

For the rest of this section, we focus on the
sequence~\eqref{eq:intro2}. Our goal is to prove the following. 

\begin{thm}
  \label{thm:exseq2}
  Let $\pi_\infty$ be a uniformizer at $\infty$, and let
  $\partial^2_\infty$ and $t_\infty$ be the corresponding residue and
  transfer maps. For all $\qspot\in X^{(1)}_\af$, let $\pi_\qspot$,
  $s_\qspot$ be a coherent choice of uniformizer and linear
  functional, and let $\partial^2_\qspot$ and $t_\qspot$ be the
  corresponding residue and transfer maps. The following sequence is
  exact:
  \begin{equation}
  \label{eq:exseq2}
  0\lra W(D) \overset{s_D}\lra W(k) \xrightarrow{\ext_F} W(F)
  \overset{\delta}\lra \bigoplus_{\pspot\in X^{(1)}} W(\resatp)
  \xrightarrow{\sum t_\pspot} W^-(D) \lra 0.
  \end{equation}
\end{thm}

We break the proof into several steps.

\subsection{Exactness at the first three terms}

The exactness of the sequence
\begin{equation}
  \label{eq:WC}
0\lra W(X) \xrightarrow{\res_F} W(F) \overset{\delta}\lra
\bigoplus_{\pspot\in X^{(1)}} W(\resatp)
\end{equation}
is the purity property of $W(X)$. It
follows from Knebusch's general result \cite[Satz~13.3.6]{Kneb}.
The first terms of the exact sequence~\eqref{eq:exseq2} are obtained
by pasting this sequence with the following. 

\begin{prop}
  \label{prop:WC}
  The following sequence is exact:
  \[
  0\lra W(D)\overset{s_D}\lra W(k) \xrightarrow{\ext_X} W(X) \lra 0.
  \]
\end{prop}

\begin{proof}
  The scalar extension map $\ext_X\colon W(k)\to W(X)$ is known to be
  surjective; see \cite[Section~5]{Pump}. (Pumpl\"un's general result
  holds for arbitrary Severi--Brauer varieties. The case of conics is
  simpler;  see \cite[Proposition~5.3]{Pump} or \cite[Proposition~2.1]{PSS}.) Since
  restriction to the generic point 
  is an injective map $W(X)\to W(F)$ (see~\eqref{eq:WC}), the
  kernel of $\ext_X$ is also the kernel of the scalar extension map
  $W(k)\to W(F)$. The latter is the ideal generated by the norm form
  of $D$, see \cite[Corollary~X.4.28]{Lam}, which by Lemma~\ref{lem:quat}
  can also be described as the image of the injective map $s_D$. The
  proposition follows.
\end{proof}

Note that the exactness of \eqref{eq:exseq2} at $W(k)$ and $W(F)$ has
already been observed by Pfister \cite[Theorem~4]{Pf}. The rest of this
section deals with the last two terms of this sequence.

\subsection{Choice of uniformizers}
\label{subsec:unif}
Proposition~\ref{prop:coherent2} shows that for
Theorem~\ref{thm:exseq2} the coherent choice of uniformizers and
linear functionals is irrelevant. We make a specific choice as
follows. At $\infty$, we choose $x^{-1}$ as a uniformizer. To choose
uniformizers at the points $\pspot\neq\infty$, recall
from \cite[Proposition~1]{Pf} or \cite[Lemma~A.9]{MT} that
the affine ring~$\VR_\af$ is a principal ideal domain. Therefore, for each $\pspot\in
X^{(1)}_\af$, we may pick an irreducible element $\pi_\pspot\in
\VR_\af$ generating the prime ideal
$\VR_\af\cap\mathfrak{m}_\pspot$. The divisor of $\pi_\pspot$ 
is thus $\pspot+v_\infty(\pi_\pspot)\infty$; hence
\begin{equation}
  \label{eq:degpi}
  \deg\pspot=-2v_\infty(\pi_\pspot)
\end{equation}
because the degree of every principal divisor is~$0$. The element
$\pi_\pspot$ is a uniformizer at $\pspot$, and the linear functional
$s_\pspot\colon \resatp\to k$ such that $\pi_\pspot$, $s_\pspot$ is
coherent is uniquely determined. Moreover, the inclusion
$\VR_\af\subset\VR_\pspot$ induces a canonical isomorphism
$\VR_\af/\pi_\pspot\VR_\af=\resatp$.

\subsection{Nullity}
We next show that the sequence in Theorem~\ref{thm:exseq2} is a zero
sequence. Since we already know that it is exact at $W(D)$, $W(k)$,  and
$W(F)$, it suffices to prove 
\[
\sum_\pspot
t_\pspot\left(\partial^2_\pspot(\qf f)\right)=0\quad\text{for all $f\in
F^\times$.} 
\]
We may assume $f\in\VR_\af$ is square-free; hence
$f=c\pi_1\cdots\pi_n$
for some $c\in k^\times$ and some pairwise
distinct irreducible elements $\pi_1$, \ldots, $\pi_n$ of $\VR_\af$
selected in Section~\ref{subsec:unif}. Since $t_\pspot$ and
$\partial^2_\pspot$ are $W(k)$-linear, we may moreover assume
$c=1$. Thus, for the rest of this subsection we fix
\[
f=\pi_1\cdots\pi_n\in\VR_\af.
\]
We let $\pspot_1$, \ldots, $\pspot_n\in X^{(1)}_\af$ be the closed
points corresponding to $\pi_1$, \ldots, $\pi_n$.

The primary decomposition of the $\VR_\af$-module
$f^{-1}\VR_\af/\VR_\af$ is
\begin{equation}
  \label{eq:primdec}
  f^{-1}\VR_\af/\VR_\af = (\pi_1^{-1}\VR_\af/\VR_\af)
  \oplus\cdots\oplus (\pi_n^{-1}\VR_\af/\VR_\af).
\end{equation}
Multiplication by $f$ defines an isomorphism $f^{-1}\VR_\af/\VR_\af
\simeq \VR_\af/f\VR_\af$; likewise, multiplication by $\pi_\alpha$
defines an isomorphism $\pi_\alpha^{-1}\VR_\af/\VR_\af\simeq
\VR_\af/\pi_\alpha\VR_\af = k(\pspot_\alpha)$ for $\alpha=1$, \ldots,
$n$. Hence we have an isomorphism of $\VR_\af$-modules $\Phi$ which
makes the following diagram commute:
\[
\xymatrix{f^{-1}\VR_\af/\VR_\af \ar[r]^-{\sim} \ar[d]_{f\cdot} &
(\pi_1^{-1}\VR_\af/\VR_\af)
\oplus\cdots\oplus (\pi_n^{-1}\VR_\af/\VR_\af)
\ar[d]^{(\pi_1\cdot)\oplus\cdots\oplus(\pi_n\cdot)} \\
\VR_\af/f\VR_\af \ar[r]^-{\Phi}&
k(\pspot_1)\oplus\cdots\oplus k(\pspot_n)\rlap{.}
}
\]
In contrast with the isomorphism provided by the Chinese remainder
theorem, the map $\Phi$ is \emph{not} a ring homomorphism; it is
readily verified that for $f_1$, \ldots, $f_n\in\VR_\af$,
\begin{equation}
  \label{eq:Phinv}
  \Phi^{-1}\left(f_1(\pspot_1),\ldots,
  f_n(\pspot_n)\right) = 
  (f_1\pi_2\cdots\pi_n)+\cdots+(\pi_1\cdots\pi_{n-1}f_n) +f\VR_\af.
\end{equation}

Recall that $v_\infty$ denotes the (normalized) valuation at $\infty$
of $F$. We have $v_\infty(g)\leq0$ for all $g\in\VR_\af$.

\begin{lem}
  \label{lem:S}
  Every element in $\VR_\af\,/f\VR_\af$ can be represented in the form
  $g+f\VR_\af$ with $g\in\VR_\af$ such that $v_\infty(g)\geq
  v_\infty(f)$. There is a well-defined $k$-linear map
  \[
  S\colon \VR_\af\,/f\VR_\af\lra k 
  \]
  such that 
  \[
  S(g+f\VR_\af) =
    -\omega_\infty\left(\frac gf(\infty)\right)\quad\text{for }
      g\in\VR_\af \text{ with } 
    v_\infty(g)\geq v_\infty(f).
  \]
  The following diagram, where $s_{\pspot_1}$, \ldots, $s_{\pspot_n}$
  are the linear functionals coherently chosen with the uniformizers
  $\pi_1$,~\ldots, $\pi_n$, commutes:
  \begin{equation}
    \label{eq:triangleS}
    \xymatrix{\VR_\af\,/f\VR_\af \ar[rr]^-{\Phi}\ar[dr]_{S} & &
      k(\pspot_1)\oplus\cdots\oplus k(\pspot_n)
      \ar[dl]^{s_{\pspot_1}+\cdots +s_{\pspot_n}}\\
    &k\rlap{.}&
    }
  \end{equation}
\end{lem}

\begin{proof}
  Proposition~\ref{prop:diffdesc} shows that every element in
  $\pi_\alpha^{-1}\VR_\af/\VR_\af$ can be represented in the form
  $g_\alpha+\VR_\af$ for some $g_\alpha\in\pi_\alpha^{-1}\VR_\af$ such
  that $v_\infty(g_\alpha)\geq0$. Therefore, by~\eqref{eq:primdec},
  every element in $\VR_\af/f\VR_\af$ has the form $(\sum_\alpha
  fg_\alpha)+f\VR_\af$, where $fg_\alpha\in\VR_\af$ and
  $v_\infty(fg_\alpha) \geq v_\infty(f)$ for all $\alpha=1$, \ldots,
  $n$. This proves the first statement.

  Note that the representation as $g+f\VR_\af$ with $g\in\VR_\af$ such
  that $v_\infty(g)\geq v_\infty(f)$ is not unique,  but if $g_1$,
  $g_2\in\VR_\af$ are such that $v_\infty(g_1)$, $v_\infty(g_2)\geq
  v_\infty(f)$ and $g_1+f\VR_\af=g_2+f\VR_\af$, then
  $f^{-1}g_1-f^{-1}g_2\in \VR_\af$ and
  $v_\infty(f^{-1}g_1-f^{-1}g_2)\geq0$, so $f^{-1}g_1-f^{-1}g_2\in
  k$. Since $\omega_\infty$ vanishes on $k$, it follows that
  $\omega_\infty(\frac{g_1}f(\infty)) =
  \omega_\infty(\frac{g_2}f(\infty))$; hence the map $S$ is
  well defined. 

  Since the choice of the functionals $s_{\pspot_1}$, \ldots,
  $s_{\pspot_n}$ is coherent with the choice of uniformizers $\pi_1$,
  \ldots, $\pi_n$, commutativity of~\eqref{eq:triangleS} amounts to
the commutativity of the following diagram:
  \[
  \xymatrix{
  f^{-1}\VR_\af\,/\VR_\af \ar[r]^-{\lowsim} \ar[d] & \bigoplus_{\alpha=1}^n
  \left(\pi_\alpha^{-1}\VR_\af\,/\VR_\af\right) \ar[d]^{\sum_\alpha
    \omega_{\pspot_\alpha}} \\
  \VR_\af\,/f\VR_\af\ar[r]^-{S}&k\rlap{.}
  }
  \]
  This readily follows from the description of the maps
  $\omega_{\pspot_\alpha}$ in Proposition~\ref{prop:diffdesc}.
\end{proof}

Tensoring $\Phi$ with the identity on the $\VR_\af$-module
$\overline{\T_\af}$, 
we obtain an isomorphism of right $\D_\af$-modules
\[
\Phi_\T\colon \overline{\T_\af}\,/f\overline{\T_\af} \overset{\lowsim}{\lra}
\overline{T(\pspot_1)}\oplus\cdots\oplus\overline{T(\pspot_n)}.
\]
On the other hand, tensoring $S$ with the identity on $D$, we obtain a
map
\[
S_D\colon \D_\af\,/f\D_\af \lra D.
\]
Recall from Proposition~\ref{prop:afsec} that
$\overline{\T_\af}=e\D_\af$. Define 
a skew-hermitian form
\[
H\colon \left(\overline{\T_\af}\,/f\overline{\T_\af}\right)\times
  \left(\overline{\T_\af}\,/f\overline{\T_\af}\right) \lra \D_\af\,/f\D_\af 
\]
by
\[
H\left(e\xi+f\overline{\T_\af}, e\eta+f\overline{\T_\af}\right) = \overline\xi
e\eta +f\D_\af 
\quad\text{for }\xi, \eta\in\D_\af.
\]

\begin{prop}
  \label{prop:isoPhi}
  The map $\Phi_\T$ is an isometry of skew-hermitian $D$-modules
  \[
  (S_D)_*(H)\simeq
  t_{\pspot_1}(\qf{\pi_2\cdots\pi_n(\pspot_1)})
  \perp\cdots\perp
  t_{\pspot_n}(\qf{\pi_1\cdots\pi_{n-1}(\pspot_n)}).
  \]
\end{prop}

\begin{proof}
  This follows by a straightforward calculation,
  using~\eqref{eq:Phinv} and Lemma~\ref{lem:S}.
\end{proof}

Next, we determine a base of $\overline{\T_\af}\,/f\overline{\T_\af}$ as
a right $D$-vector space.

\begin{lem}
  \label{lem:Qbase}
  If\, $v_\infty(f)=-n$, then
  $(ex^\alpha+f\overline{\T_\af})_{\alpha=0}^{n-1}$ is 
  a $D$-base of\, $\overline{\T_\af}\,/f\overline{\T_\af}$.
\end{lem}

\begin{proof}
  For every $\pspot\in X^{(1)}$, the fiber $\overline\Tatp$ is a
  $2$-dimensional 
  right ideal of $\Datp$; hence $\dim_{\resatp}\overline\Tatp=2$,  and
  therefore $\dim_k\overline\Tatp=2\deg\pspot$. By~\eqref{eq:degpi} it
  follows that $\dim_k\overline\Tatp=-4v_\infty(\pi_\pspot)$ for all
  $\pspot\in X_\af^{(1)}$. Since $\Phi_\T$ is an isomorphism of
  $k$-vector spaces, we get
  \[
  \dim_k\left(\overline\T_\af\,/f\overline\T_\af\right)=-4\left(v_\infty(\pi_1)+\cdots
  +v_\infty(\pi_n)\right)= -4v_\infty(f),
  \]
  hence $\dim_D(\overline{\T_\af}\,/f\overline{\T_\af})=n$. Therefore,
  to prove the lemma, it suffices to show that the sequence
  $(ex^\alpha+f\overline{\T_\af})_{\alpha=0}^{n-1}$ spans
  $\overline{\T_\af}\,/f\overline{\T_\af}$.

  From the description of $\VR_\af$ as $k[x,y]$ with $y^2=ax^2+b$, we
  know that $(x^\alpha,x^\alpha y)_{\alpha=0}^\infty$ is a $k$-base of~$\VR_\af$.
  As $v_\infty(x^\alpha)=v_\infty(x^{\alpha-1}y)=-\alpha$,
  the elements $g\in\VR_\af$ such that $v_\infty(g)\geq -n$ are linear
  combinations of $(x^\alpha)_{\alpha=0}^n$ and $(x^\beta
  y)_{\beta=0}^{n-1}$. Therefore, Lemma~\ref{lem:S} shows that
  $\VR_\af\,/f\VR_\af$ is $k$-spanned by the images of
  $(x^\alpha)_{\alpha=0}^n$ and $(x^\beta y)_{\beta=0}^{n-1}$. It
  follows that $\overline{\T_\af}\,/f\overline{\T_\af}$ is $D$-spanned
  by the images of 
  $(ex^\alpha)_{\alpha=0}^n$ and $(ex^\beta y)_{\beta=0}^{n-1}$. By
  multiplying~\eqref{eq:esq0} on the left by $(ij)^{-1}$, we get
 $je-xie+ye=0$; hence after conjugation
  \begin{equation}
    \label{eq:ey}
    ey=ej-eix\quad\text{in }\overline{\T_\af}.
  \end{equation}
  Therefore, the elements $ex^\beta y$ for $\beta=0$, \ldots, $n-1$ are in
  the $D$-span of $(ex^\alpha)_{\alpha=0}^n$. Thus, it only remains to
  see that the image of $ex^n$ in
  $\overline{\T_\af}\,/f\overline{\T_\af}$ lies in the 
  $D$-span of the image of $(ex^\alpha)_{\alpha=0}^{n-1}$.

  For this, note that since $v_\infty(f)=-n$, we have
  \[
  f=c_1x^n+c_2x^{n-1}y+f_0
  \]
  for some $f_0\in\VR_\af$ such that $v_\infty(f_0)\geq-n+1$ and some
  $c_1$, $c_2\in k$ not both~$0$, hence
  \[
  ec_1x^n+ef_0\equiv -ec_2x^{n-1}y\quad\bmod f\overline{\T_\af}.
  \]
  Comparing with~\eqref{eq:ey}, we obtain
  \[
  ec_1x^n+ef_0 \equiv -ejc_2x^{n-1}+eic_2x^n \quad\bmod f\overline{\T_\af},
  \]
  hence
  \[
  e(c_1-ic_2)x^n\equiv -ejc_2x^{n-1}-ef_0\quad\bmod f\overline{\T_\af}.
  \]
  Note that since $v_\infty(f_0)\geq-n+1$, the arguments above show
  that the image of $ef_0$ in $\overline{\T_\af}\,/f\overline{\T_\af}$
  lies in the $D$-span of 
  $(ex^\alpha)_{\alpha=0}^{n-1}$. Since the quaternion $c_1-ic_2$ is
  invertible in $D$, the proof is complete.
\end{proof}

We may now prove that the sequence in Theorem~\ref{thm:exseq2} is a
zero sequence.

\begin{prop}
  \label{prop:null}
  For $f$ as above, $\sum_{\pspot\in
    X^{(1)}}t_\pspot\left(\partial^2_\pspot(\qf f)\right)=0$.
\end{prop}

\begin{proof}
  Let $v_\infty(f)=-n$. For $\alpha$, $\beta=0$, \ldots, $n-1$, we have
  \[
  H\left(ex^\alpha+f\overline{\T_\af}, ex^\beta+f\overline{\T_\af}\right) =
  ex^{\alpha+\beta}+f\D_\af;  
  \]
  hence, as $e=bi+axj+yij$,
  \begin{equation*}
  (S_D)_*(H)\left(ex^\alpha+f\overline{\T_\af}, ex^\beta+f\overline{\T_\af}\right) = 
  bS\left(x^{\alpha+\beta}+f\VR_\af\right) i + aS\left(x^{\alpha+\beta+1}+f\VR_\af]\right)j +
  S\left(x^{\alpha+\beta}y+f\VR_\af\right)ij.
  \end{equation*}
  Now, for $g\in\VR_\af$ such that $v_\infty(g)>-n$,
  \[
  S(g+f\VR_\af) = -\omega_\infty\left(\frac gf(\infty)\right) = 0; 
  \]
  hence $ex^\alpha+f\overline{\T_\af}$ and
  $ex^\beta+f\overline{\T_\af}$ are orthogonal for 
  $(S_D)_*(H)$ when $\alpha+\beta+1<n$. In particular, the images of
  $ex^\alpha$ for $\alpha<\frac{n-1}2$ span a totally isotropic
  subspace of $\overline{\T_\af}\,/f\overline{\T_\af}$. If $n$ is even,
  this totally isotropic 
  subspace has $D$-dimension $\frac
  n2=\frac12\dim_D(\overline{\T_\af}\,/f\overline{\T_\af})$, 
  so $(S_D)_*(H)$ is hyperbolic. By Proposition~\ref{prop:isoPhi}, 
  it follows that $\sum_{\pspot\in X^{(1)}_\af}
  t_\pspot\left(\partial^2_\pspot(\qf f)\right)=0$. Since
  $\partial^2_\infty(\qf f)=0$, the proposition follows.

  Now suppose  $n=2m+1$ for some integer $m$. Then the image of
  $(ex^\alpha)_{\alpha=0}^{m-1}$ spans a totally isotropic subspace in
  the orthogonal complement of $ex^m+f\overline{\T_\af}$, so
  $(S_D)_*(H)$ is 
  Witt-equivalent to its restriction to the span of
  $ex^m+f\overline{\T_\af}$. Computation shows 
  \[
  (S_D)_*(H)\left(ex^m+f\overline{\T_\af}, ex^m+f\overline{\T_\af}\right) =
  -a\omega_\infty\left(\frac{x^n}f(\infty)\right)j -
  \omega_\infty\left(\frac{x^{n-1}y}f(\infty)\right)ij; 
  \]
  hence by Proposition~\ref{prop:isoPhi}
  \begin{equation}
    \label{eq:null1}
    \sum_{\pspot\in X^{(1)}_\af} t_\pspot\left(\partial^2_\pspot(\qf
    f)\right) =
    \left\langle-a\omega_\infty\left(\frac{x^n}f(\infty)\right)j - 
  \omega_\infty\left(\frac{x^{n-1}y}f(\infty)\right)ij\right\rangle
  \quad\text{in $W^-(D)$.} 
  \end{equation}
  On the other hand, since $n$ is odd, we have
  \[
  \qf f = \left\langle{x^{-n-1}f}\right\rangle = \left\langle{x^{-1}}\right\rangle \left\langle{f^{-1}x^n}\right\rangle,
  \]
  hence $\partial_\infty^2(\qf f) = \qf{\frac{x^n}f(\infty)}$
  in $W(\resatinf)$. By Proposition~\ref{prop:explicitransfer}, the
  Witt class $t_\infty\left(\partial^2_\infty(\qf f)\right)$ is
  represented by the transfer along $s_{D_\infty}$ of the skew-hermitian
  form $h_f$ on $\overline\Tatinf$ such that
  \[
  h_f(e_\infty,e_\infty)=\frac{x^n}f(\infty) e_\infty.
  \]
  As observed in the proof of Lemma~\ref{lem:Qbase},
  $\overline\Tatinf$ is a 
  $D$-vector space of dimension~$1$. Taking $e_\infty$ as a base of
  $\overline\Tatinf$, we obtain
  \begin{equation}
    \label{eq:null2}
    t_\infty\left(\partial_\infty^2(\qf f)\right) =
    \left\langle s_{D_\infty}\left(\frac{x^n}f(\infty)
    e_\infty\right)\right\rangle. 
  \end{equation}
  Recall that $e_\infty$ is the image in $\Tatinf$ of $ex^{-1}$,
  so $e_\infty=aj+\frac{y}x(\infty) ij$. Therefore,
  \begin{equation}
    \label{eq:null3}
    s_{D_\infty}\left(\frac{x^n}f(\infty) e_\infty\right) =
    a\omega_\infty\left(\frac{x^n}f(\infty)\right)j +
    \omega_\infty\left(\frac{x^{n-1}y}f(\infty)\right)ij. 
  \end{equation}
  The proof follows by comparing~\eqref{eq:null1}, \eqref{eq:null2},
  and \eqref{eq:null3}.
\end{proof}

\subsection{Exactness at $\boldsymbol{\bigoplus_\pspot W(\resatp)}$}
\label{subsec:exact}

We prove the exactness of the sequence in Theorem~\ref{thm:exseq2} by
relating it to the following exact sequence due to Pfister
\cite[Theorem~5]{Pf}:
\begin{equation}
  \label{eq:Pf}
  W(F)\overset{\delta''}\lra \bigoplus_{\pspot\in X^{(1)}_\af}
  W(\resatp) \xrightarrow{\sum(s_\pspot)_*} W(k)/J, 
\end{equation}
where $J=\{\varphi\in W(k)\mid \qf{1,-a}\varphi=0\}$ is the subgroup
annihilated by the norm form of $\resatinf$ and $\delta''$ is the map
whose $\pspot$-component is $\partial_\pspot^2$ for all $\pspot\in
X^{(1)}_\af$, for a coherent choice of uniformizer and linear
functional at each $\pspot$. For this, we use the canonical
isomorphism $\resatinf\simeq k(i)\subset D$ of
\cite[Proposition~45.12]{EKM}. Let $K=k(i)$, and let $\gamma$ be the
canonical isomorphism
\[
\gamma\colon \resatinf \overset{\lowsim}{\lra} K,\quad \frac yx(\infty)
\longmapsto i.
\]
Theorem~\ref{thm:octa} yields an exact sequence
\[
W^-(D)\overset{\pi_2}\lra W(K) \overset{\sigma_2}\lra W^-(D)
\overset{\pi_1}\lra W^-(K,\invo).
\]

\begin{lem}
  \label{lem:t=sigma}
  Let $\Psi\colon W(\resatinf)\to W(K)$ be the isomorphism that
  maps every quadratic form $\qf g$ to $\qf{-\overline{\gamma(g)}}$,
  for $g\in \resatinf^\times$. The following diagram commutes:
  \[
  \xymatrix{
  W(\resatinf) \ar[dr]_{t_\infty} \ar[rr]^{\Psi} & &
  W(K)\ar[dl]^{\sigma_2}\\ 
  &W^-(D)\rlap{.}&
  }
  \]
\end{lem}

\begin{proof}
  For $g\in \resatinf^\times$, the definitions of $\sigma_2$ and $\Psi$ yield
  \[
  \sigma_2\left(\Psi(\qf g)\right) = \left\langle{-ij\overline{\gamma(g)}}\right\rangle
  =\qf{-\gamma(g)ij}, 
  \]
  whereas
  \[
  t_\infty(\qf g) = \qf{s_{D_\infty}(ge_\infty)} =
  \left\langle a\omega_\infty(g)j + \omega_\infty\left(g\,\frac
  yx(\infty)\right)ij\right\rangle 
  \]
  (see \eqref{eq:null2} and \eqref{eq:null3}). Computation yields
  \[
  a\omega_\infty(g) + \omega_\infty\left(g\,\frac
  yx(\infty)\right)i=-\gamma(g)i, 
  \]
  hence 
  \[
  t_\infty(\qf g) = \qf{-\gamma(g)ij} = 
  \sigma_2\circ\Psi(\qf g).
  \qedhere
  \]
\end{proof}

\begin{lem}
  \label{lem:Theta}
  The map that carries every symmetric bilinear form $\varphi\colon
  U\times U\to k$ to the skew-hermitian form 
  $
  \qf{bi}\varphi_{(K,\invo)}\colon (U\otimes_kK)\times
  (U\otimes_kK)\to K
  $
  such that for $u$, $u'\in U$ and $\alpha$, $\alpha'\in K$, 
  \[
  \qf{bi}\varphi_{(K,\invo)}(u\otimes\alpha,
    u'\otimes\alpha') = bi\overline{\alpha}\,\varphi(u,u')\alpha'
  \] 
  induces a group isomorphism
  \[
  \Theta\colon W(k)/J \overset{\lowsim}{\lra} W^-(K,\invo).
  \]
  This map makes the following diagram commute for every $\pspot\in
  X^{(1)}_\af$:
  \[
  \xymatrix{
  W(\resatp) \ar[r]^{t_\pspot} \ar[d]_{(s_\pspot)_*} & W^-(D)
  \ar[d]^{\pi_1} \\
  W(k)/J\ar[r]^-{\Theta}& W^-(K,\invo)\rlap{.}
  }
  \]
\end{lem}

\begin{proof}
  Multiplication by $bi$ defines an isomorphism $W^-(K,\invo)\simeq
  W(K,\invo)$; hence for every symmetric bilinear form
  $\varphi\colon U\times U\to k$,  the skew-hermitian form
  $\qf{bi}\varphi_{(K,\invo)}$ is hyperbolic if and only if the
  hermitian form $\varphi_{(K,\invo)}$ is hyperbolic. This occurs if
  and only if the quadratic form $s_K(\varphi_{(K,\invo)})\colon
  U\otimes_kK\to k$ defined by
  $s_K(\varphi_{(K,\invo)})(u\otimes\alpha) = \overline{\alpha}\,
  \varphi(u,u)\alpha$ is hyperbolic; see \cite[Section~10.1, p.~348]{Sch}.
  Since
  $s_K(\varphi_{(K,\invo)})= \qf{1,-a}\varphi$, it follows that the
  map $\Theta$ is well defined and injective. It is also surjective
  because every skew-hermitian form over $(K,\invo)$ has a
  diagonalization $\qf{\alpha_1i,\ldots,\alpha_ni}$ with $\alpha_1$,
  \ldots, $\alpha_n\in k^\times$.

  Now, let $f\in \resatp^\times$. We know from
  Proposition~\ref{prop:explicitransfer} that the Witt class of
  $t_\pspot(\qf f)$ is represented by the skew-hermitian form
  \[
  h\colon \overline\Tatp\times \overline\Tatp\lra D \quad\text{defined by}\quad
    h(e_\pspot\xi, e_\pspot \eta)=s_{\Datp}\left(f\overline\xi e_\pspot
    \eta\right) \quad\text{for }\xi, \eta\in \Datp.
  \]
  By multiplying \eqref{eq:esq0} on the left by $(ij)^{-1}$ and by
  $j^{-1}$, we obtain 
  \[
    je-xie+ye=0 \quad\text{and}\quad -ije+axe-yie=0; 
  \]
  hence, after conjugation,
  \[
  ej=ey+eix \quad\text{and}\quad eij=-eax-eiy \quad\text{in
    $\overline{\T_\af}$.}
  \]
  Therefore, $e_\pspot j$ and $e_\pspot ij$ are in the $\resatp$-span of
  $e_\pspot$ and $e_\pspot i$ in $\overline\Tatp$; hence $(e_\pspot,
  e_\pspot i)$ is a $\resatp$-base of $\overline\Tatp$. Let
  $(c_\alpha)_{\alpha=1}^{\deg\pspot}$ be a $k$-base of
  $\resatp$. Then $(e_\pspot c_\alpha, e_\pspot i
  c_\alpha)_{\alpha=1}^{\deg\pspot}$ is a $k$-base of $\overline\Tatp$; 
  hence $(e_\pspot c_\alpha)_{\alpha=1}^{\deg\pspot}$ is a $K$-base of
  $\overline\Tatp$. Since $e_\pspot=bi+ax(\pspot)j+y(\pspot)ij$, we have
  \[
  s_{\Datp}\left(fc_\alpha e_\pspot c_\beta\right) = s_\pspot\left(fc_\alpha
  c_\beta\right)bi + s_\pspot\left(fc_\alpha x(\pspot) c_\beta\right)aj + s_\pspot\left(f
  c_\alpha y(\pspot) c_\beta\right)ij; 
  \]
  hence the matrix of $\pi_1(h)$ in the base $(e_\pspot
  c_\alpha)_{\alpha=1}^{\deg\pspot}$ is
  \[
  \left(s_\pspot\left(fc_\alpha
  c_\beta\right)bi\right)_{\alpha,\beta=1}^{\deg\pspot}.
  \]
  The skew-hermitian form
  $\qf{bi}(s_\pspot)_*\left(\qf{f}\right)_{(K,\invo)}$ has the same
  matrix. 
\end{proof}

In the next lemma, we use the following notation: We write
\[
\varpi\colon \bigoplus_{\pspot\in X^{(1)}}W(\resatp) \lra
\bigoplus_{\pspot\in X^{(1)}_\af}W(\resatp)
\]
for the map that ``forgets'' the component at $\infty$.

\begin{lem}
  \label{lem:forget}
We have  $\ker \varpi\cap \ker(\sum t_\pspot) \subset \image(\delta)$.
\end{lem}

\begin{proof}
  Let $(\varphi_\pspot)_{\pspot\in X^{(1)}}\in\ker\varpi\cap\ker(\sum
  t_\pspot)$. Thus, $\varphi_\pspot=0$ for all $\pspot\neq\infty$ and
  $t_\infty(\varphi_\infty)=0$. Lemma~\ref{lem:t=sigma} yields
  $\sigma_2\left(\Psi(\varphi_\infty)\right)=0$; hence by
  Theorem~\ref{thm:octa}, we may find an $h\in W^-(D)$ such that
  $\pi_2(h)=\Psi(\varphi_\infty)$. From the description of $\pi_2$
  in~\eqref{eq:defpi2}, it follows that $\varphi_\infty$ is a sum of
  Witt classes represented by quadratic forms of the type $\qf u
  \qf{1,-a\lambda^2-b N_{\resatinf/k}(u)}$, for $u\in K^\times$ and
  $\lambda\in k$. To complete the proof, it suffices to show that
  every element in $\bigoplus_{\pspot\in X^{(1)}} W(\resatp)$ whose
  $\pspot$-components are~$0$ for all $\pspot\neq\infty$ and whose
  $\infty$-component is represented by a form of the type above is in
  the image of $\delta$.

  Fix $\lambda\in k$ and $u=u_1+ u_2\frac yx(\infty)\in K^\times$
  (with $u_1$, $u_2\in k$), and consider
  \[
  f=\lambda+u_1x+u_2y\in\VR_\af.
  \]
  Since $u_1$ and $u_2$ are not both zero, we have $v_\infty(f)=-1$; 
  hence $f$ is irreducible in $\VR_\af$. Let $\qspot\in X^{(1)}_\af$
  be the point such that $f\in \mathfrak{m}_\qspot\cap\VR_\af$. Then
  \[
  \resatq\simeq k\left(\sqrt{a\lambda^2+bu_1^2-abu_2^2}\,\right); 
  \] 
  hence (for any choice of uniformizer at $\qspot$)
  \[
  \partial_\qspot^2\left(\qf f
  \qf{1,-a\lambda^2-bN_{\resatinf/k}(u)}\right) = 0 \quad\text{in
    $W(\resatq)$.}
  \]
  Moreover, for every $\pspot\in X^{(1)}_\af$ with $\pspot\neq\qspot$,
  \[
  \partial^2_\pspot\left(\qf f
  \qf{1,-a\lambda^2-bN_{\resatinf/k}(u)}\right) = 0 \quad\text{in
    $W(\resatp)$}
  \]
  because $v_\pspot(f)=0$. Furthermore, with $x^{-1}$ as
  uniformizer at $\infty$, 
  \[
  \partial^2_\infty\left(\qf f
  \qf{1,-a\lambda^2-bN_{\resatinf/k}(u)}\right) = \qf u
  \qf{1,-a\lambda^2-bN_{\resatinf/k}(u)} \quad\text{in
    $W(\resatinf)$.}
  \]
  Thus, the element in $\bigoplus_{\pspot\in X^{(1)}} W(\resatp)$
  whose $\pspot$-components are all $0$ for $\pspot\neq\infty$ and
  whose $\infty$-component is represented by $\qf u
  \qf{1,-a\lambda^2-bN_{\resatinf/k}(u)}$ is the image of $\qf f
  \qf{1,-a\lambda^2-bN_{\resatinf/k}(u)}$ under $\delta$.
\end{proof}

\begin{prop}
  \label{prop:exactatsum}
  The sequence \eqref{eq:exseq2} is exact at
  $\bigoplus_\pspot W(\resatp)$.
\end{prop}

\begin{proof}
  Consider the following diagram:
  \[
  \xymatrix@C=35pt{
  W(F) \ar@{=}[d] \ar[r]^-{\delta} & \bigoplus_{\pspot\in X^{(1)}}
  W(\resatp) \ar[d]_{\varpi} \ar[r]^-{\sum t_\pspot} & W^-(D)
  \ar[d]^{\Theta^{-1}\circ\pi_1} \\
  W(F) \ar[r]^-{\delta''} & \bigoplus_{\pspot\in X^{(1)}_\af}
  W(\resatp) \ar[r]^-{\sum(s_\pspot)_*} & W(k)/J\rlap{.}
  }
  \]
  The left square commutes by the definition of the maps, and the right
  square commutes by Lemma~\ref{lem:Theta}. The upper sequence is a
  zero sequence by Proposition~\ref{prop:null}, and the lower sequence
  is exact by Pfister's theorem \cite[Theorem~5]{Pf}. Therefore, a diagram
  chase yields for every $u\in\ker(\sum t_\pspot)$ an element $v\in
  W(F)$ such that $\delta''(v)=\varpi(u)$. Then $u-\delta(v)\in\ker\varpi$,
  and $u-\delta(v)\in\ker(\sum t_\pspot)$ because the upper sequence
  is a zero sequence. Therefore, Lemma~\ref{lem:forget} shows that
  $u-\delta(v)\in\image(\delta)$, hence $u\in\image(\delta)$.
\end{proof}

\subsection{Exactness at $W^-(D)$}

To complete the proof of Theorem~\ref{thm:exseq2}, we show that the
map $\sum t_\pspot$ is onto. It suffices to prove that $1$-dimensional
skew-hermitian forms over $D$ are in the image of $\sum t_\pspot$.

\begin{prop}
  \label{prop:onto}
  For every nonzero $q\in D^0$,  there exist a $\pspot\in X^{(1)}_\af$ of
  degree~$2$ 
  and an $f\in \resatp^\times$ such that $t_\pspot(\qf f)= \qf q$.
\end{prop}

\begin{proof}
  Let $q=\lambda_1i+\lambda_2j+\lambda_3ij$ with $\lambda_1$,
  $\lambda_2$, $\lambda_3\in k$. We may find $\alpha_1$, $\alpha_2$,
  $\alpha_3\in k$ with $\alpha_2$, $\alpha_3$ not both zero such that
  \[
  \alpha_1\lambda_1+\alpha_2\lambda_2+\alpha_3\lambda_3=0.
  \]
  Let $\pspot\in X^{(1)}$ be the intersection of the conic with the
  line $\alpha_1bZ+\alpha_2aX+\alpha_3Y=0$ in the projective plane
  $\mathbb{P}(D^0)$. The point $\pspot$ has degree~$2$, and
  $\pspot\neq\infty$ since $\alpha_2$ and $\alpha_3$ are not both
  zero. In $\resatp$, the following equation holds:
  \[
  \alpha_1b+\alpha_2ax(\pspot) +\alpha_3y(\pspot)=0.
  \]
  Therefore, there is a linear functional $r\colon \resatp\to k$ such
  that 
  \[
  r(b)=\lambda_1, \quad r(ax(\pspot))=\lambda_2,\quad\text{and}\quad
  r(y(\pspot))=\lambda_3.
  \]
  Every $k$-linear functional on $\resatp$ has the form $g\mapsto
  s_\pspot(fg)$ for some $f\in \resatp$;  hence we may find an $f\in
  \resatp^\times$ such that $r(g)=s_\pspot(fg)$ for all $g\in
  \resatp$. The element $e_\pspot=bi+ax(\pspot)j +y(\pspot)ij$ is a
  $D$-base of $\overline\Tatp$, and Proposition~\ref{prop:explicitransfer}
  shows that in this base
  \[
  t_\pspot(\qf f) = \qf{s_{\Datp}(fe_\pspot)} = \qf{s_\pspot(fb)i +
    s_\pspot(fax(\pspot))j + s_\pspot(fy(\pspot))ij} = \qf q.
  \qedhere
  \]
\end{proof}

The proof of Theorem~\ref{thm:exseq2} is thus complete.

%%%%%%%%%%%%%%%%%%%%%
% References
%%%%%%%%%%%%%%%%%%%%%

\newcommand{\etalchar}[1]{$^{#1}$}

\end{document}